\documentclass[12pt, reqno]{amsart}
\usepackage{graphicx,cite}
\usepackage{amsfonts,amsmath,amssymb}
\usepackage{theorem,verbatim}
\usepackage{color}

\advance \topmargin by -\headheight
\advance \topmargin by -\headsep
\evensidemargin \oddsidemargin
\newcounter{theorem}
\def\openthm#1#2{\refstepcounter{theorem}\bigskip
{\noindent\bf#1~\thetheorem\if#2!{. }\else{ (#2).}\fi}
\it}
\def\thmskip{}

\newenvironment{lemma}[1][!]{\openthm{Lemma}{#1}}{\thmskip}

\newenvironment{proposition}[1][!]{\openthm{Proposition}{#1}}{\thmskip}

\newcounter{definition}
\def\opendefn#1#2{\refstepcounter{definition}\bigskip
{\noindent\bf#1~\thetheorem\if#2!{. }\else{ (#2).}\fi}
\it}
\newenvironment{definition}[1][!]{\openthm{Definition}{#1}}{\thmskip}

\newcounter{remark}
\def\openrem#1#2{\refstepcounter{remark}\bigskip
{\noindent \it \bfseries#1~\theremark\if#2!{. }\else{ (#2). }\fi}}
\newenvironment{remark}[1][!]{\openrem{Remark}{#1}}{\qed}

\def\ilist{\renewcommand{\labelenumi}{(\roman{enumi})}}

\def\R{\mathbb{R}}
\def\N{\mathbb{N}}
\def\Z{\mathbb{Z}}

\def\CC{{\rm C}}

\def\dt{\,{\rm d}t}

\def\<{\langle}
\def\>{\rangle}

\def\eps{\varepsilon}

\def\E{\mathcal{E}}

\def\Ys{\mathcal{Y}}
\def\Us{\mathcal{U}}
\def\As{\mathcal{A}}
\def\Cs{\mathcal{C}}

\renewcommand{\cases}[1]{\left\{ \begin{array}{rl} #1 \end{array} \right.}

\newcommand{\smfrac}[2]{{\textstyle \frac{#1}{#2}}}

\newcommand{\alert}[1]{{#1}}

\begin{document}

\title[Accuracy of Quasicontinuum Approximations Near Instabilities]{
  Accuracy of Quasicontinuum Approximations \\ Near
  Instabilities}

\author{M. Dobson}
\author{M. Luskin}
\address{School of Mathematics, 206 Church St. SE,
  University of Minnesota, Minneapolis, MN 55455, USA}
\email{dobson@math.umn.edu, luskin@umn.edu}

\author{C. Ortner} \address{Mathematical Institute, 24--29 St. Giles',
  Oxford OX1 3LB, UK} \email{ortner@maths.ox.ac.uk}

\date{\today}


\thanks{ This work was supported in part by DMS-0757355, DMS-0811039,
  the Department of Energy under Award Number DE-FG02-05ER25706, the
  Institute for Mathematics and Its Applications, the University of
  Minnesota Supercomputing Institute, the University of Minnesota
  Doctoral Dissertation Fellowship, and the EPSRC critical mass
  programme ``New Frontier in the Mathematics of Solids''
  (OxMoS). The project was initiated during a visit of ML to OxMoS. \\
  {\color{white} .} \quad  To appear in J. Mech. Phys. Solids.}


\begin{abstract}
  The formation and motion of lattice defects such as cracks,
  dislocations, or grain boundaries, occurs when the lattice
  configuration loses stability, that is, when an eigenvalue of the
  Hessian of the lattice energy functional becomes negative. When the
  atomistic energy is approximated by a hybrid energy that couples
  atomistic and continuum models, the accuracy of the approximation
  can only be guaranteed near deformations where both the atomistic
  energy as well as the hybrid energy are stable.  We propose,
  therefore, that it is essential for the evaluation of the predictive
  capability of atomistic-to-continuum coupling methods near
  instabilities that a theoretical analysis be performed, at least for
  some representative model problems, that determines whether the
  hybrid energies remain stable {\em up to the onset of instability of
  the atomistic energy}.

  We formulate a one-dimensional model problem with nearest and
  next-nearest neighbor interactions and use rigorous analysis,
  asymptotic methods, and numerical experiments to obtain such sharp
  stability estimates for the basic conservative quasicontinuum (QC)
  approximations.   Our results show that the consistent quasi-nonlocal QC
approximation correctly reproduces the stability of the atomistic system,
  whereas the inconsistent energy-based QC approximation
  incorrectly predicts instability at a significantly reduced applied load
  that we describe by an analytic criterion in
  terms of the derivatives of the atomistic potential.

\end{abstract}

\maketitle

\section{Introduction}

An important application of atomistic-to-continuum coupling methods is
the study of the quasistatic deformation of a crystal in order
to model instabilities such as dislocation
formation during nanoindentation, crack growth, or the deformation of
grain boundaries~\cite{Miller:2003a}. In each of these applications,
the quasistatic evolution provides an accurate approximation of the
crystal deformation until the evolution approaches an unstable
configuration. This occurs, for example, when a dislocation forms or
moves or when a crack tip advances.  The crystal will then typically
undergo a dynamic process until it reaches a new stable configuration.
In order to guarantee an accurate approximation of the entire
quasistatic crystal deformation, up to the formation of an
instability, it is crucial that the equilibrium in the
atomistic/continuum hybrid method is stable whenever the corresponding
atomistic equilibrium is. The purpose of this work is to investigate
whether the quasicontinuum (QC) method has this property. In technical
terms, this requires sharp estimates on the stability constant in the
QC approximation.

The QC method is an atomistic-to-continuum coupling method that models
the continuum region by using an energy density that exactly
reproduces the lattice-based energy density at uniform strain (the
Cauchy-Born rule)~\cite{Miller:2003a,Ortiz:1995a,Shenoy:1999a}.
Several variants of the QC approximation have been
proposed that differ in how the atomistic and continuum regions are
coupled~\cite{Dobson:2008a,E:2006,Miller:2003a,Shimokawa:2004}.  In
this paper, we present sharp stability analyses for the main examples
of conservative QC approximations as a means to evaluate their
relative predictive properties for defect formation and motion.  Our
sharp stability analyses compare the loads for which the atomistic
energy is stable, that is, those loads where the Hessian of the
atomistic energy is positive-definite, with the loads for which the
QC energies are stable.  It has previously been suggested and
then observed in computational experiments that
inconsistency at the atomistic-to-continuum interface can reduce the accuracy
for computing a critical applied load~\cite{Shenoy:1999a,Miller:2003a,Miller:2008,E:2006}.
In this paper, we give an analytical method to estimate
the error in the critical applied load by deriving stability
criteria in terms of  the derivatives of the atomistic potential.

Although we present our techniques
in a precise mathematical format, we believe that these techniques
can be utilized in a more informal way by computational scientists to
quantitatively evaluate the predictive capability of other atomistic-to-continuum or
multiphysics models as they arise.
For example, our quantitative approach
has the potential to estimate the reduced critical applied load
in QC approximations such as
the quasi-nonlocal QC
approximation (QNL), that are consistent
for next-nearest interactions but not for longer range interactions.
Since the longer range interactions are generally weak, such
an estimate may give an analytical basis to judging that the
reduced critical applied load for QNL with finite range interactions
is within an acceptable error tolerance.

The accuracy of various QC approximations and other
atomistic-to-continuum coupling methods is currently being
investigated by both computational experiments and numerical analysis
~\cite{BadiaParksBochevGunzburgerLehoucq:2007,Legoll:2005,Dobson:2008c,Dobson:2008b,dobs-qcf2,E:2005a,Gunzburger:2008a,Gunzburger:2008b,LinP:2003a,LinP:2006a,luskin-cluster-2008,mingyang,Ortner:2008a,PrudhommeBaumanOden:2005}.
The main issue that has been studied to date in the mathematical
analyses is the rate of convergence with respect to the smoothness of
the continuum solution (however,
see~\cite{Legoll:2005,dobs-qcf2,Ortner:2008a} for
analyses of the error of the QC solutions with respect to the
atomistic solution, possibly containing defects).  Some error
estimates have been obtained that give theoretical justification for
the accuracy of a QC approximation for all loads up to the
critical atomistic load where the atomistic model loses
stability~\cite{Dobson:2008b,Ortner:2008a}, but
other error estimates that have been presented do not hold near the
atomistic limit loads. It is important to understand whether the
break-down of these error estimates is an artifact of the analysis, or
whether the particular QC approximation actually does incorrectly
predict an instability before the applied load has reached the correct
limit load of the atomistic model.

Two key ingredients in any approximation error analysis are the
consistency and stability of the approximation scheme. For energy
minimization problems, consistency means that the truncation error for
the equilibrium equations is small in a suitably chosen norm, and
stability is usually understood as the positivity of the Hessian of
the functional. For the highly non-convex problems we consider here,
stability must necessarily be a local property: The configuration
space can be divided into stable and unstable regions, and the
question we ask is whether the stability regions of different QC
approximations approximate the stability region of the full atomistic model
in a way that can be controlled in the setup of the method (for
example, by a judicious choice of the atomistic region).

In this work, we initiate such a systematic study of the stability of
QC approximations. In the present paper, we investigate
conservative QC approximations, that is, QC approximations which are formulated in
terms of the minimization of an energy functional. In a companion
paper \cite{doblusort:qcf.stab}, we study the stability of a
force-based approach to atomistic-to-continuum coupling that is
nonconservative.

In computational experiments, one often studies the evolution of a
system under incremental loading. There, the critical load at which
the system ``jumps'' from one energy well to another is often the goal
of the computation. Thus, we will also study the effect of the
``stability error'' on the error in the critical load.

We will formulate a simple model problem, a one dimensional periodic
atomistic chain with pairwise next-nearest neighbour interactions of
Lennard-Jones type, for which we can analyze the issues layed out in
the previous paragraphs. It is well known that the uniform
configuration is stable only up to a critical value of the tensile
strain (fracture). We use analytic, asymptotic, and numerical
approaches to obtain sharp results for the stability of different
QC approximations when applied to this simple model.

In Section~\ref{sec:model}, we describe the model and the various
QC approximations that we will analyze.  In
Section~\ref{sec:analysis}, we study the stability of the atomistic
model as well as two consistent QC approximations: the local QC
approximation (QCL) and the quasi-nonlocal QC approximation (QNL).  We prove that
the critical applied strains for both of these approximations are equal to
the critical applied strain for the atomistic model, up to
second-order in the atomistic spacing.

A similar analysis for the inconsistent QCE approximation is more difficult
because the uniform configuration is not an equilibrium. Thus, in
Section~\ref{sec:ana:qce}, we construct a first-order correction of
the uniform configuration to approximate an equilibrium configuration,
and we study the positive-definiteness of the Hessian for the
linearization about this configuration.  We explicitly construct a
test function with strain concentrated in the atomistic-continuum
interface that is unstable for applied strains bounded well away from
the atomistic critical applied strain.

In Section~\ref{sec:discus}, we analyze the accuracy in predicting the
critical strain for onset of instability. For the QCL and QNL approximations,
this involves comparing the effect of the difference between their
modified stability criteria and that of the atomistic model.  For QCE,
since the solution to the nonlinear equilibrium equations are
non-trivial, we provide computational results in addition to an
analysis of the critical QCE strain predicted by the approximations
derived in Section \ref{sec:ana:qce}.

\section{The atomistic and quasicontinuum models}
\label{sec:model}

\subsection{The atomistic model problem}
\label{sec:intro:model_problem}
\begin{figure}[t]
  \hspace{-10mm} \includegraphics[height=5cm]{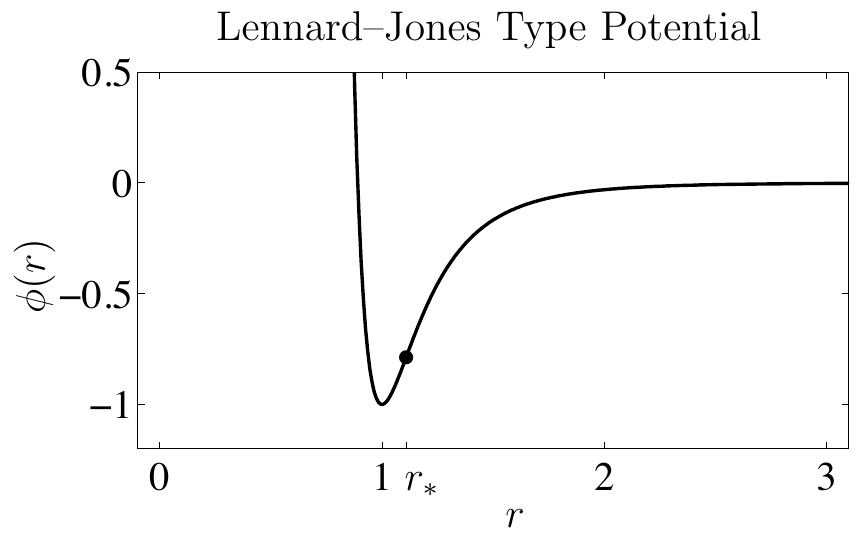} \\[-4mm]
  \caption{\label{fig:lj}Lennard-Jones type interaction potential. The
    bond length $r_*$ is the turning point between the convex and
    concave regions of $\phi$.}
\end{figure}
Suppose that the infinite lattice $\eps \Z$ is deformed uniformly into
the lattice $y_F := F\eps  \Z,$ where $F > 0$ is the macroscopic
deformation gradient and where $\eps > 0$ scales
the reference atomic spacing, that is,
\begin{displaymath}
  (y_F)_\ell:=F\ell\eps \quad\text{for }-\infty<\ell<\infty.
\end{displaymath}
We admit $2N$-periodic perturbations $u = (u_\ell)_{\ell \in \Z}$ from
the uniformly deformed lattice $y_F.$ More precisely, for fixed $N \in \N$, we
admit deformations $y$ from the space
\begin{equation*}
  \Ys_F := \big\{ y \in \R^\Z : y = y_F + u,\,
  u \in \Us \big\},
\end{equation*}
where $\Us$ is the space of $2N$-periodic displacements with zero mean,
\begin{equation*}
  \Us := \big\{ u \in \R^\Z : u_{\ell+2N} = u_\ell \text{ for }
  \ell \in \Z, \text{ and } {\textstyle \sum_{\ell = -N+1}^N}
  u_\ell = 0 \big\}.
\end{equation*}
We set $\eps = 1/N$ throughout so that the reference length of the periodic
domain is fixed. Even though the energies and forces we
will introduce are well-defined for all $2N$-periodic displacements,
we require that they have zero mean in order to obtain locally unique
solutions to the equilibrium equations.  These zero mean constraints
are an artifact of our periodic boundary conditions and are similarly
used in the analysis of continuum problems with periodic boundary conditions.

We assume that the {\em stored energy per period} of a deformation $y
\in \Ys_F$ is given by a next-nearest neighbour pair interaction
model,
\begin{equation}\label{bond}
  \E_{\rm a}(y) := \eps \sum_{\ell = -N+1}^N \big(\phi(y_\ell')
  + \phi(y_\ell' + y_{\ell+1}')\big),
\end{equation}
where $v_\ell'$ is the backward difference
\begin{displaymath}
  v_\ell' := \eps^{-1}(v_\ell - v_{\ell-1})\qquad\text{for }v \in \R^\Z,\
  \ell \in \Z,
\end{displaymath}
and where $\phi$ is a Lennard-Jones type
interaction potential satisfying (see also Figure \ref{fig:lj})
\begin{enumerate}
  \ilist
\item $\phi \in \CC^4((0, +\infty); \R)$,
\item there exists $r_* > 0$ such that $\phi$ is convex in $(0, r_*)$
  and concave in $(r_*, +\infty)$, and
\item $\phi^{(k)}(r) \rightarrow 0$ rapidly as $r \nearrow \infty$,
  for $k = 0, \dots, 4$.
\end{enumerate}
We have used the scaled interaction potential, \alert{$\eps\phi(r/\eps),$} in the
definition of the stored energy, $\E_{\rm a}(y),$ to obtain a
continuum limit as $\eps\to 0.$ Assumptions (i) and (ii) are used
throughout our analysis, while assumption (iii) serves primarily to
motivate that next-nearest neighbour interaction terms are typically
dominated by nearest-neighbour terms. Note, however, that even with
assumption (iii), the relative size of next-nearest and nearest
neighbour interactions is comparable when strains approach $r_*$.

We denote the first variation of the energy functional, $\E_{\rm
  a}'(y)[u],$ at a deformation $y\in\Ys_F$ by
\begin{equation*}
  \E_{\rm a}'(y)[u] :=\sum_{\ell = -N+1}^N \frac{\partial \E_{\rm a}(y)}{\partial y_\ell}u_{\ell}
  =\eps \sum_{\ell = -N+1}^N \Big\{\phi'(y_\ell') u_\ell'
  +\phi'(y_\ell'+y_{\ell+1}')(u_\ell' + u_{\ell+1}')\Big\},
\end{equation*}
for $u \in \Us$. In the absence of external forces, the uniformly
deformed lattice $y = y_F$ is an equilibrium of the atomistic energy
under perturbations from $\Us$, that is,
\begin{equation}
  \label{eq:DEa_yF}
  \E_{\rm a}'(y_F)[u] = 0 \qquad \text{ for all } u \in \Us.
\end{equation}

We identify the stability of $y_F$ with {\em linear stability under
  perturbations from the space $\Us$}.  To make this precise, we
denote the second variation of the energy functional, $\E_{\rm a}''(y)[u,v],$ evaluated at a deformation
$y\in\Ys_F$, by
\begin{equation}
  \label{eq:defn_Ea_hess}
  \begin{split}
  \E_{\rm a}''(y)[u,v] &:=\sum_{\ell,\, m = -N+1}^N  \frac{\partial^2 \E_{\rm a}(y)}{\partial y_\ell \partial y_m
  }u_{\ell}v_m \\
  &= \eps \sum_{\ell = -N+1}^N \Big\{
  \phi''(y_\ell') u_\ell'  v_\ell'
  + \phi''(y_\ell' + y_{\ell+1}') [u_\ell' + u_{\ell+1}'] [v_\ell' + v_{\ell+1}']  \Big\},
  \end{split}
\end{equation}
for all $u,v\in\Us$. The matrix $\Big(\frac{\partial^2 \E_{\rm
    a}(y)}{\partial y_\ell \partial y_m }\Big)_{\ell, m = -N+1}^N$ is
the Hessian for the energy functional.  We say that the equilibrium
$y_F$ is {\em stable} for the atomistic model if this Hessian,
evaluated at $y = y_F$, is positive definite on the subspace $\Us$ of
zero mean displacements, or equivalently, if
\begin{equation}\label{hess}
  \E_{\rm a}''(y_F)[u,u] > 0 \qquad \text{ for all } u \in \Us \setminus\{0\}.
\end{equation}
\alert{In Section~\ref{sec:summary}, Definition~\ref{def:stab}, we
  extend this definition of stability to the various QC approximation
  and their equilibria.}

Note that if $y = y_F$, then $y_\ell' = F$ and $y_\ell' + y_{\ell+1}' =
2F$ for all $\ell$. Therefore, upon defining the quantities
\begin{displaymath}
  \phi_F'' := \phi''(F), \quad
  \phi_{2F}'' := \phi''(2F), \quad \text{and} \quad
  A_F = \phi_F'' + 4 \phi_{2F}'',
\end{displaymath}
we can rewrite \eqref{eq:defn_Ea_hess} as follows
\begin{equation}
  \label{eq:Ea_hess_F}
  \E_{\rm a}''(y_F)[u,u] = \eps \sum_{\ell = -N+1}^N \big\{ \phi_F'' |u_\ell'|^2
  + \phi_{2F}'' |u_\ell' + u_{\ell+1}'|^2 \big\} \quad \text{for } u\in\Us.
\end{equation}
(We will use $A_F$ later.)  The quantities $\phi_F''$ and
$\phi_{2F}''$ will play a prominent role in the analysis of the
stability of the atomistic model and its QC approximations and
describe the strength of the nearest neighbor and next-nearest neighbor
interactions, respectively.  We
similarly define the quantities $\phi^{(k)}_G$ for all $k \in \N$ and
for all $G > 0$.  For most realistic interaction potentials the
second-nearest neighbour coefficient is non-positive, $\phi_{2F}''
\leq 0,$ except in the case of extreme compression (see Figure
\ref{fig:lj}).  Therefore, in order to avoid having to distinguish
several cases, we will assume throughout our analysis that $F \geq
r_*/2$. In this case, property (ii) of the interaction potential shows
that $\phi_{2F}'' \leq 0$.

We also note that, for $u \in \Us$, both $u'$ and $u''$ are understood
as $2N$-periodic chains, that is, $u', u'' \in \Us,$ where
the centered second
difference $u'' \in \Us$ is defined by
\begin{displaymath}
  u_\ell'' := \eps^{-2}(u_{\ell+1} - 2 u_\ell + u_{\ell-1})
  \qquad\text{for }u \in \R^\Z,\
  \ell \in \Z.
\end{displaymath}
For $u, v \in \Us$, we also define the weighted $\ell^p$-norms
\begin{displaymath}
  \| v \|_{\ell^p_\eps} := \cases{
    \displaystyle \Bigg( \sum_{\ell=-N+1}^N \eps |v_\ell|^p \Bigg)^{1/p}, &
    1 \leq p < \infty, \\[10pt]
    \displaystyle \max_{\ell = -N+1, \dots, N} |v_\ell|, & p = \infty,
  }
\end{displaymath}
as well as the weighted $\ell^2$-inner product
\begin{displaymath}
  \< u, v \> = \eps \sum_{\ell = -N+1}^N u_\ell v_\ell.
\end{displaymath}

\subsection{The local QC approximation (QCL)}
\label{sec:model:qcl}
Before we introduce different flavors of QC
approximations, we note that we can rewrite the atomistic energy as a
sum over the contributions from each atom,
\begin{align*}
  \E_{\rm a}(y) =~& \eps \sum_{\ell = -N+1}^N E^a_\ell(y)
  \qquad \text{where} \\
  E^a_\ell(y) :=~& \smfrac12 \big[ \phi(y_\ell')
  + \phi(y_{\ell+1}')
  + \phi(y_{\ell-1}' + y_\ell')
  + \phi(y_{\ell+1}'+y_{\ell+2}') \big].
\end{align*}

If $y$ is ``smooth,'' i.e., $y_\ell'$ varies slowly, then $E^a_\ell(y)
\approx E^c_\ell(y)$ where
\begin{displaymath}
  E^c_\ell(y) := \smfrac12 \big[ \phi(y_\ell') + \phi(y_{\ell+1}')
  + \phi(2y_\ell') + \phi(2y_{\ell+1}') \big]
  = \smfrac12 \big[ \phi_{cb}(y_\ell') + \phi_{cb}(y_{\ell+1}') \big],
\end{displaymath}
and where $\phi_{cb}(r) := \phi(r) + \phi(2r)$ is the so-called {\em
  Cauchy-Born stored energy density}. In this case, we may expect
that the atomistic model is accurately represented by the local QC (or
continuum) model
\begin{equation}\label{qcldef}
  \E_{\rm qcl}(y) := \eps \sum_{\ell = -N+1}^N E^c_\ell(y)
  = \eps \sum_{\ell = -N+1}^N \phi_{cb}(y_\ell').
\end{equation}
The main feature of this {\em continuum} model is that the
next-nearest neighbour interactions have been replaced by nearest
neighbour interactions, thus yielding a model with more {\em
  locality}. Such a model can subsequently be coarse-grained (i.e.,
degrees of freedom are removed) which yields efficient numerical
methods.

\subsection{The energy-based QC approximation (QCE)}
\label{sec:model:qce}
If $y_\ell'$ is ``smooth'' in the majority of the computational
domain, but not in a small neighbourhood, say, $\{ -K, \dots, K \}$,
where $K > 1$, then we can obtain sufficient accuracy and efficiency by coupling the atomistic model to the local QC
model by simply choosing energy contributions $E^a_\ell$ in the {\em
  atomistic region} $\As = \{-K, \dots, K\}$ and $E^c_\ell$ in the
{\em continuum region} $\Cs = \{-N+1, \dots, N\} \setminus \As$. This
approximation of the atomistic energy is often called the {\em energy based QC approximation}
\cite{Ortiz:1995a} and yields the energy functional
\begin{align*}
  \E_{\rm qce}(y) :=& \eps\sum_{\ell \in \Cs}  E^c_\ell(y)
  + \eps \sum_{\ell \in \As} E^a_\ell(y).
\end{align*}

\alert{ It is now well-understood
  \cite{Dobson:2008a,Dobson:2008c,Dobson:2008b,E:2006,Shenoy:1999a}
  that the QCE approximation exhibits an inconsistency (``ghost
  force'') near the interface, which is displayed in the fact that
  $\E_{\rm qce}'(y_F)\ne 0 .$ The first remedy of this lack of
  consistency was the {\em ghost force correction
    scheme}~\cite{Shenoy:1999a} which eventually led to the derivation
  of the force-based QC approximation \cite{Dobson:2008a} and which we
  analyze in \cite{doblusort:qcf.stab} and \cite{dobs-qcf2}. }

\subsection{Quasi-nonlocal coupling (QNL)}
\label{sec:model:qnl}
An alternative approach was suggested in \cite{Shimokawa:2004}, which
requires a modification of the energy at the interface. This idea is
best understood in terms of interactions rather than energy
contributions of individual atoms (see also \cite{E:2006} where this
has been extended to longer range interactions). The nearest neighbour
interactions are left unchanged. A next-nearest neighbour interaction
$\phi(\eps^{-1}(y_{\ell+1}-y_{\ell-1}))$ is left unchanged if at least
one of the atoms $\ell+1, \ell-1$ belong to the atomistic region and
is replaced by a Cauchy--Born approximation,
\begin{displaymath}
  \phi(\eps^{-1}(y_{\ell+1}-y_{\ell-1})) \approx
  \smfrac12 \big[ \phi(2y_\ell') + \phi(2y_{\ell+1}')]
\end{displaymath}
if {\em both} atoms belong to the continuum region. This idea leads to
the energy functional
\begin{equation*}
  \begin{split}
    \E_{\rm qnl}(y) :=\eps \sum_{\ell = -N+1}^N  \phi(y_\ell')
    + \eps \sum_{\ell \in \As_{\rm qnl}} \phi(y_\ell' + y_{\ell+1}')
    + \eps \sum_{\ell \in \Cs_{\rm qnl}} \smfrac12 \big[ \phi(2y_\ell')
    + \phi(2y_{\ell+1}')\big]
  \end{split}
\end{equation*}
where $\As_{\rm qnl} = \{-K-1, \dots, K+1\}$ and $\Cs_{\rm qnl} =
\{-N+1,\dots, N\} \setminus \As_{\rm qnl}$ are modified atomistic and
continuum regions.  The QNL approximation is consistent, that is, $y = y_F$
is an equilibrium of the QNL energy functional. The label QNL comes
from the original intuition of considering interfacial atoms as {\em
  quasi-nonlocal}, i.e., they interact by different rules with atoms
in the atomistic and continuum regions.

\alert{
\section{Stability of Quasicontinuum Approximations: Summary}
\label{sec:summary}
In this section, we briefly summarize the our main results.

We begin by giving a careful definition of a notion of stability. Our
condition is slightly stronger than local minimality, which is the
natural concept of stability in statics. However, an analysis of local
minimality alone is usually not tractable. Moreover, for the
deformations that we consider, our definition is in fact sufficiently
general.

\begin{definition}[Stable Equilibrium]\label{def:stab}
  Let $\E : \Ys_F \to \R \cup \{+\infty\}$. We say that $y \in \Ys_F$
  is a stable equilibrium of $\E$ if $\E$ is twice differentiable at
  $y$ and the following conditions hold:
  \begin{itemize}
  \item[{\it (i)}] \quad $\E'(y)[u] = 0$ for all $u \in \Us$,
  \item[{\it (ii)}] \quad $\E''(y)[u,u] > 0$ for all $u \in \Us \setminus \{0\}$.
  \end{itemize}
  If only (i) holds, then we call $y$ a critical point of $\E$.
\end{definition}

\medskip \noindent We focus on the deformation $y_F =
(F\ell\eps)_{\ell \in \Z}$ and ask for which macroscopic strains $F$
this deformation is a stable equilibrium. We know from
\eqref{eq:DEa_yF} that $y_F$ is a critical point of the atomistic
energy $\E_{\rm a}$, and it is easy to see that $y_F$ is also a
critical point of the QCL energy $\E_{\rm qcl}$ and of the QNL energy
$\E_{\rm qnl}$. Our analysis in Section \ref{sec:analysis} gives the
following conditions under which $y_F$ is stable:
\begin{enumerate}
\item[1.] $y_F$ is a stable equilibrium of $\E_{\rm a}$ if and only if
  $A_F - \eps^2 \pi^2 \phi_{2F}'' + O(\eps^4) > 0$;

  \vspace{1mm}
\item[2.] $y_F$ is a stable equilibrium of $\E_{\rm qcl}$ if and only
  if $A_F > 0$;

  \vspace{1mm}
\item[3.] $y_F$ is a stable equilibrium of $\E_{\rm qnl}$ if and only
  if $A_F > 0$
\end{enumerate}
where we recall that $A_F = \phi_F'' + 4 \phi_{2F}''$ is the continuum
elastic modulus for the Cauchy--Born stored energy function
$\phi_{cb}(r) = \phi(r) + \phi(2r)$. Points 1., 2., and 3. are
established, respectively, in Propositions \ref{th:ana:stab_a},
\ref{th:ana:qcl_stab}, and \ref{th:ana:qnl_stab}.

If we envision a quasistatic process in which $F$ is slowly increased,
then we may wish to find the critical strain $F^*$ at which $y_F$ is
no longer a stable equilibrium (fracture instability). If we denote
the critical strains in the atomistic, QCL, and QNL models,
respectively, by $F_{\rm a}^*$, $F_{\rm qcl}^*$ and $F_{\rm qnl}^*$,
then 1.--3. imply that (cf. Section \ref{sec:discus})
\begin{displaymath}
  |F_{\rm a}^* - F_{\rm qcl}^*| = O(\eps^2) \quad \text{and} \quad
  |F_{\rm a}^* - F_{\rm qnl}^*| = O(\eps^2).
\end{displaymath}

For the QCE approximation defined in Section \ref{sec:model:qce}, the
situation is more complicated. The occurrence of a ``ghost force'' in
the QCE model implies that $y_F$ is {\em not} a critical point of
$\E_{\rm qce}$, and consequently, we will need to analyze the
stability of the second variation $\E_{\rm qce}''(y_{{\rm qce}, F})$
where $y_{{\rm qce}, F} \neq y_F$ is an appropriately chosen
equilibrium of $\E_{\rm qce}$. Since $y_{{\rm qce}, F}$ solves a
nonlinear equation, we will replace it by an approximate equilibrium
in our analysis in Section \ref{sec:ana:qce} where we obtain the
following (simplified) result:
\begin{enumerate}
\item[4.] For $y_{{\rm qce}, F}$ to be a stable equilibrium of
  $\E_{\rm qce}$ it is necessary that
  \begin{displaymath}
    1 + \frac{3 \phi_{2F}''}{2 \phi_F''}
    + \frac{\phi_F''' \phi_{2F}'}{2 |\phi_F''|^2} + O(\delta^2) > 0,
  \end{displaymath}
  where $\delta = \max\{ |\phi^{(j)}(2F) / \phi''(F)|:  j = 1,2,3
  \}$ is assumed to be small.
\end{enumerate}
We remark that 4. gives only a necessary but not a sufficient
condition for stability of the QCE equilibrium $y_{{\rm qce}, F}$,
which, moreover, depend on assumptions on the parameter $\delta$. We
refer to Remark \ref{rem:delta_discussion} for a careful discussion
of the role of $\delta$.

If we let $\tilde{F}_{\rm qce}^*$ denote the critical strain at which
4. fails (ignoring the $O(\delta^2)$ term), then we obtain
\begin{displaymath}
  | F_{\rm a}^* - \tilde{F}_{\rm qce}^* | = O(1),
\end{displaymath}
which suggests that the QCE method is unable to predict the onset of
fracture instability accurately. In Section \ref{sec:discus}, we
confirm this asymptotic prediction with numerical experiments.

We have shown in \cite{qcf.iterative} that the stability properties of
the ghost force correction scheme (GFC) can be understood for uniaxial
tensile loading by considering the stability of the QC energy
\begin{equation}\label{gfc}
  \E_{{\rm gfc},\, y_F}(y):=\E_{\rm qce}(y)-\E_{\rm qce}'(y_F)(y-y_F)
  \quad\text{for all }y \in \Ys_F.
\end{equation}
We note that $\E_{{\rm gfc},\, y_F}'(y_F)=0,$ so $y_F$ is an
equilibrium of the $\E_{{\rm gfc},\, y_F}$ energy under perturbations
from $\Us.$ We can therefore analyze the stability of $\E_{{\rm
    gfc},\, y_F}(y)$ at $y_F$ by studying the Hessian $\E_{\rm
  qce}''(y_F)=\E_{{\rm gfc},\, y_F}''(y_F).$ We show in
Remark~\ref{rmbounds} in Section~\ref{sec:model:qce} that
\begin{enumerate}
\item[5.] $y_F$ is a stable equilibrium of $\E_{{\rm gfc},\, y_F}$ if and only if
  $A_F + \lambda_K
  \phi_{2F}'' > 0$ where $\smfrac 12 \le \lambda_K \le 1.$
\end{enumerate}
We give analytic and computation results in Sections~\ref{sec:ana:qce} and \ref{sec:discus} showing that
the ghost force correction scheme
can be expected to improve the accuracy of the computation of the critical strain
by the QCE method,
that is,
\[
\tilde{F}_{\rm qce}^*<F_{\rm qce}^{y_F}<F_{\rm a}^*,
\]
where $F_{\rm qce}^{y_F}$ is the critical strain
      at which $\E_{\rm qce}''(y_F)=\E_{{\rm gfc},\, y_F}''(y_F)$ is no longer positive definite, but the
      error in computing the critical strain by the GFC scheme is still
\begin{displaymath}
  | F_{\rm a}^* - F_{\rm qce}^{y_F} | = O(1).
\end{displaymath}
}

\section{Sharp Stability Analysis of Consistent QC Approximations}
\label{sec:analysis}
In this section, we analyze the stability of the atomistic model and
two consistent QC approximations: the local QC approximation and the
quasi-nonlocal QC approximation. In each case, we will give precise
conditions on $F$ under which $y_F$ is stable in the respective
approximation. The inconsistent energy-based QC approximation (QCE) is
analyzed in Section \ref{sec:ana:qce}. The corresponding result for
QCE is less exact than for QCL and QNL, but shows that there is a much
more significant loss of stability.

\subsection{Atomistic model}
\label{sec:ana:a}
Recalling the representation of $\E_{a}''(y_F)$ from
\eqref{eq:Ea_hess_F} and noting that
\begin{equation}
  \label{eq:rewrite_nnn}
   |u_\ell'+u_{\ell+1}'|^2 = 2|u_\ell'|^2 + 2|u_{\ell+1}'|^2 -  |u_{\ell+1}' - u_\ell'|^2,
\end{equation}
we obtain
\begin{align}
  \notag
  \E_{a}''(y_F)[u,u]
  =~& \eps\sum_{\ell = -N+1}^N \phi_F'' |u_\ell'|^2
  + \eps \sum_{\ell = -N+1}^N \phi_{2F}'' \big(2|u_\ell'|^2 + 2|u_{\ell+1}'|^2 -  |u_{\ell+1}' - u_\ell'|^2\big) \\
  \notag
  =~& \eps\sum_{\ell = -N+1}^N (\phi_{F}'' + 4 \phi_{2F}'') |u_\ell'|^2
  + \eps \sum_{\ell = -N+1}^N (-\eps^2\phi_{2F}'')|u_\ell''|^2 \\
  \label{eq:Eq_second_diff_form}
  =~&  A_F \| u' \|_{\ell^2_\eps}^2 + (-\eps^2 \phi_{2F}'') \|u''\|_{\ell^2_\eps}^2.
\end{align}

To quantify the influence of the strain gradient term, we define
\begin{displaymath}
  \mu_{\eps} := \inf_{\psi \in \Us \setminus\{0\}}
  \frac{ \| \psi'' \|_{2} }{ \| \psi'\|_2 }.
\end{displaymath}
Since $u$ is periodic, it follows that $u'$ has zero mean. In this
case, the eigenvalue $\mu_{\eps}$ is known to be attained by the
eigenfunction $\psi'_\ell=\sin(\eps \ell\pi)$ and is given by
\cite[Exercise 13.9]{SuliMayers}
\begin{equation}
  \label{eq:defn_muaeps}
  \mu_{\eps} = \frac{2 \sin(\pi\eps/2)}{\eps}.
\end{equation}
Since $\sin(t) = t + O(t^3)$ as $t \searrow 0$, it follows that
$\mu_\eps = \pi + O(\eps^2)$ as $\eps \searrow 0$. Thus, we obtain the
following stability result for the atomistic model.

\begin{proposition}\label{th:ana:stab_a}
  Suppose $\phi_{2F}'' \leq 0$.  Then $y_F$ is stable in the atomistic
  model if and only if $A_F - \eps^2 \mu_{\eps}^{2} \phi_{2F}'' >
  0$, where $\mu_{\eps}$ is the eigenvalue defined in
  \eqref{eq:defn_muaeps}.
\end{proposition}
\begin{proof}
  By the definition of $\mu_{\eps}$, and using
  \eqref{eq:Eq_second_diff_form}, we have
  \begin{displaymath}
    \inf_{\substack{u \in \Us \\ \|u'\|_{\ell^2_\eps} = 1}} \E_{\rm a}''(y_F)[u,u]
    = A_F - \eps^2 \phi_{2F}''
    \inf_{\substack{u \in \Us \\ \|u'\|_{\ell^2_\eps} = 1}} \|u''\|_{\ell^2_\eps}^2
    = A_F - \eps^2 \mu_\eps^2 \phi_{2F}''. \qedhere
  \end{displaymath}
\end{proof}

\subsection{The Local QC approximation}
The equilibrium system, in variational form, for the QCL approximation is
\begin{displaymath}
  \E_{\rm qcl}'(y)[u] = \eps \sum_{\ell = -N+1}^N \big( \phi'(y_\ell') + 2\phi'(2 y_\ell') \big) u_\ell' = 0 \qquad \text{ for all } u \in \Us.
\end{displaymath}
Since $u'$ has zero mean, it follows that $y = y_F$ is a critical
point of $\E_{\rm qcl}$ for all $F$.  The second variation of the
local QC energy, evaluated at $y = y_F$, is given by
\begin{align*}
  \E_{\rm qcl}''(y_F)[u,u] =~& \eps \sum_{\ell = -N+1}^N A_F |u_\ell'|^2
  \qquad \text{ for } u \in \Us.
\end{align*}
Thus, recalling our
definition of stability from Section \ref{sec:intro:model_problem}, we
obtain the following result.

\begin{proposition}\label{th:ana:qcl_stab}
  The deformation $y_F$ is a stable equilibrium of the local QC approximation
  if and only if $A_F > 0$.
\end{proposition}

\medskip Comparing Proposition \ref{th:ana:qcl_stab} with Proposition
\ref{th:ana:stab_a} we see a first discrepancy, albeit small, between
the stability of the full atomistic model and the local QC approximation (or
the Cauchy--Born approximation). In Section \ref{sec:discus} we will
show that this leads to a negligible error in the computed critical
load.

\subsection{Quasi-nonlocal coupling}

By the construction of the QNL coupling rule at the interface,
the deformation $y = y_F$ is an equilibrium of $\E_{\rm qnl}$
\cite{Shimokawa:2004}.  The second variation of $\E_{\rm qnl}$ evaluated at $y =
y_F$ is given by
\begin{equation*}
  \begin{split}
    \E_{\rm qnl}''(y_F)[u,u] = \eps\sum_{\ell = -N+1}^N \phi_F'' |u_\ell'|^2
    +~& \eps \sum_{\ell \in \As_{\rm qnl}} \phi_{2F}'' |u_\ell' + u_{\ell+1}'|^2  \\
    +~& \eps \sum_{\ell \in \Cs_{\rm qnl}} 4 \phi_{2F}''
    ( \smfrac12 |u_\ell'|^2 + \smfrac12 |u_{\ell+1}'|^2 ).
  \end{split}
\end{equation*}

We use \eqref{eq:rewrite_nnn} to rewrite the second group on the
right-hand side (the nonlocal interactions) in the form
\begin{displaymath}
  \eps \sum_{\ell = -K-1}^{K+1}  \phi_{2F}'' |u_{\ell}'+u_{\ell+1}'|^2
  = \eps \sum_{\ell = -K-1}^{K+1} \big(2\phi_{2F}'' (|u_\ell'|^2 + |u_{\ell+1}'|^2)
  - \eps^2 \phi_{2F}'' |u_\ell''|^2\big),
\end{displaymath}
to obtain
\begin{equation*}
  \E_{\rm qnl}''(y_F)[u,u] = \eps \sum_{\ell = -N+1}^{N} A_F |u_\ell'|^2
  + \eps \sum_{\ell = -K-1}^{K+1} (-\eps^2 \phi_{2F}'') |u_\ell''|^2.
\end{equation*}
Except in the case $K \in \{N-1, N\}$, it now follows immediately that
$y_F$ is stable in the QNL approximation if and only if $A_F > 0$.

\begin{proposition}\label{th:ana:qnl_stab}
  Suppose that $K < N-1$ and that $\phi_{2F}'' \leq 0$, then $y_F$ is
  stable in the QNL approximation if and only if $A_F > 0$.
\end{proposition}

\section{Stability Analysis of the Energy-based QC approximation}
\label{sec:ana:qce}
We will explain in Remark~\ref{rmbounds} that there exists $\smfrac 12
\le \lambda_K \le 1$ such that
\begin{displaymath}
  \E_{\rm qce}''(y_F) \quad
  \text{is positive definite if and only if}\quad
  A_F + \lambda_K \phi_{2F}'' > 0.
\end{displaymath}
However, $y_F$ is not a critical point of $\E_{\rm qce},$ so we must
be careful in extending the previous definition of stability to the
QCE approximation.
We cannot simply consider the positive-definiteness of
$\E_{\rm qce}''(y_F).$
Instead, we analyze the second variation $\E_{\rm qce}''(y_{{\rm qce}, F})$
where $y_{{\rm qce}, F} \in \Ys_F$ solves the QCE equilibrium equations
\begin{displaymath}
  \frac{\partial \E_{\rm qce}}{\partial y_\ell} (y_{{\rm qce}, F}) = 0 \quad
  \text{for } \ell = -N+1, \dots, N,
\end{displaymath}
or equivalently
\begin{equation}
  \label{eq:qce:nonlin}
  \E_{\rm qce}'(y_{{\rm qce}, F})[u] = 0 \qquad \text{ for all } u \in \Us.
\end{equation}
We will see that, when the second-neighbour interactions are small
compared with the first neighbour interactions (which we make precise
in Lemma~\ref{th:asymptotic_expansion}), there is a locally
unique solution $y_{{\rm qce}, F}$ of the equilibrium equations, which is
the correct QCE counterpart of $y_F$.  We will then derive a stability
criterion for the equilibrium deformation $y_{{\rm qce}, F}.$

In Proposition~\ref{th:qce:final_stab_result} below, we derive an
upper bound for the coercivity of $\E_{\rm qce}''(y_{{\rm qce}, F})[u,\,u]$
with respect to the norm $\| u'\|_{\ell^2_\eps}.$ Even though the
derivation of this upper bound is only rigorous for strains bounded
away from the atomistic critical strain, it clearly identifies a source
of instability that cannot be found by analyzing, for example,
$\E_{\rm qce}''(y_F)$. Moreover, we will present numerical experiments
in Section~\ref{sec:discus} showing that the critical strain predicted in
our following analysis gives a remarkably accurate approximation to
actual QCE critical strain.

In Section \ref{sec:discus}, we consider the critical strain for
each approximation, namely the point at which the appropriate equilibrium
deformation (either $y_F$ or $y_{{\rm qce},F}$) becomes unstable.
We will see later in this
section, as well as in Section~\ref{sec:discus} that predicting the
loss of stability for the QCE approximation using $y_F$
greatly underestimates the error in approximating the atomistic
critical
strain by the QCE critical strain.

Due to the nonlinearity and nonlocality of the interaction law, we
cannot compute $y_{{\rm qce}, F}$ explicitly. Instead, we will construct an
approximation $\hat{y}_{{\rm qce},F}$ which is accurate whenever
second-neighbour terms are dominated by first-neighbour terms. In the
following paragraphs, we first present a semi-heuristic construction,
motivated by the analysis in \cite{Dobson:2008c}, and then a rigorous
approximation result, the proof of which is given in Appendix
\ref{sec:app_qce_eq}.

In (\ref{eq:app:DEqce}) in the appendix, we provide an explicit
representation of $\E_{\rm qce}'$. Inserting $y = y_F$, we obtain a
variational representation of the atomistic-to-continuum interfacial
truncation error terms that are often dubbed ``ghost forces,''
\begin{equation}
  \label{eq:qce:ghost_force}
  \begin{split}
    \E_{\rm qce}'(y_F)[u] &~= \eps \smfrac12 \phi_{2F}' \big\{
    u_{-K-1}' - u_{-K+1}' - u_K' + u_{K+2}' \big\} \\
    &~:= - \phi_{2F}' \< \hat g', u' \> \qquad \text{ for } u \in \Us,
  \end{split}
\end{equation}
where
\begin{equation}
  \label{eq:qce:defn_uhat}
  \hat g_\ell' = \cases{
    -\smfrac12, & \ell = -K-1, K+2, \\
    \smfrac12, & \ell = -K+1, K, \\
    0, & \text{otherwise}.
  }
\end{equation}
We note that (\ref{eq:qce:ghost_force}) makes our claim precise that
$y_F$ is not a critical point of $\E_{\rm qce}$.

Motivated by property (iii) of the interaction potential $\phi$, we
will assume that the parameters
\begin{displaymath}
  \delta_1 := \frac{\phi'(2F)}{\phi''(F)} \quad \text{and} \quad
  \delta_2 := \frac{- \phi''(2F)}{\phi''(F)}
\end{displaymath}
are small, and construct an approximation for $y_{{\rm qce}, F}$ which is
asymptotically of second order as $\delta_1, \delta_2 \to 0$. Although
such an approximation will not be valid near the critical strain for
the QCE approximation, it will give us a rough impression how the
inconsistency affects the stability of the system.  \alert{We note that
$\delta_1$ is scale invariant since we used a scaled interaction
potential, $\eps\phi(r/\eps),$ in our definition of the stored
energy \eqref{bond}.}

A non-dimensionalization of (\ref{eq:qce:ghost_force}) shows that
$y_{{\rm qce}, F} = y_{F} + O(\delta_1)$. If $\delta_1$ is small, then we
can linearize (\ref{eq:qce:nonlin}) about $y_F$ and find the
first-order correction $y_{\rm lin} \in \Ys_F$, which is given by
\begin{equation}
  \label{eq:qce:lin_correction}
  \E_{\rm qce}''(y_F)[ y_{\rm lin} - y_F, u] = - \E_{\rm qce}'(y_F)[u]
  = \phi_{2F}' \< \hat g', u' \> \qquad \text{ for all } u \in \Us.
\end{equation}
We note that this linear system is precisely the one analyzed in
detail in \cite{Dobson:2008c}. However, instead of using the implicit
representation of $y_{\rm lin} - y_F$ obtained there, we use the
assumption that $\delta_2$ is small to simplify
(\ref{eq:qce:lin_correction}) further and obtain a more explicit
approximation.

Writing out the bilinear form $\E_{\rm qce}''(y_F)[u, u]$ explicitly
(using (\ref{eq:app:qce_hess_2}) as a starting point) gives
\begin{equation}
  \label{eq:Eqce_decomposition_final}
  \begin{split}
    \E_{\rm qce}''(y_F)[u,u] =~& \dots
    + \eps \sum_{\ell = 0}^N \phi_F'' |u_\ell'|^2
    + \eps \sum_{\ell = 0}^{K-1} \phi_{2F}'' |u_\ell' + u_{\ell+1}'|^2
    + \eps\sum_{\ell = K+2}^N 4 \phi_{2F}'' |u_\ell'|^2 \\
    & + \smfrac{\eps}{2} \phi_{2F}''|u_K'+u_{K+1}'|^2
    + \smfrac{\eps}{2} \phi_{2F}'' |u_{K+1}'+u_{K+2}'|^2
    + \smfrac{\eps}{2} 4 \phi_{2F}'' |u_{K+1}'|^2,
  \end{split}
\end{equation}
where we have only displayed the terms in the right half of the domain
and indicated the terms in the left half by dots.  Ignoring all terms
involving $\phi_{2F}''$, which are of order $\delta_2$ relative to the
remaining terms, we arrive at the following approximation of
\eqref{eq:qce:lin_correction}:
\begin{equation*}
  \phi_F'' \big\< (\hat y_{{\rm qce}, F} - y_F)', u' \big\> = \phi_{2F}' \< \hat g', u' \>
  \qquad \text{ for all } u \in \Us,
\end{equation*}
the solution of which is given by
\begin{equation*}
  \hat y_{{\rm qce}, F} = y_F + \delta_1 \hat g.
\end{equation*}

The following lemma makes this approximation rigorous. A complete
proof is given in Appendix \ref{sec:app_qce_eq}.

\begin{lemma}
  \label{th:asymptotic_expansion}
  If $\delta_1$ and $\delta_2$ are sufficiently small,
  then there exists a (locally unique) solution $y_{{\rm qce}, F}$ of
  \eqref{eq:qce:nonlin} such that
  \begin{displaymath}
    \| (y_{{\rm qce}, F} - \hat y_{{\rm qce}, F})' \|_{\ell^\infty}
    \leq C (\delta_1^2 + \delta_1\delta_2),
  \end{displaymath}
  where $C$ may depend on $\phi$ (and its derivatives) and on $F$, but
  is independent of $\eps$.
\end{lemma}

\medskip \noindent From now on, we will also assume that $\delta_3 :=
\phi_{2F}''' / \phi_{F}''$ is small, and combine the three small
parameters into a single parameter
\begin{displaymath}
  \delta := \max(|\delta_1|, |\delta_2|, |\delta_3|).
\end{displaymath}
We will neglect all terms which are of order $O(\delta^2)$. \alert{A
  careful discussion of the parameter $\delta$ and the validity of the
  asymptotic analysis is given in Remark \ref{rem:delta_discussion}.}

In the following, we will again only show terms appearing on the right
half of the domain. Our goal in the remainder of this section is to
obtain an estimate for the smallest eigenvalue of $\E_{\rm qce}''(\hat
y_{{\rm qce}, F})$. Using (\ref{eq:app:qce_hess_2}), we can represent
$\E_{\rm qce}''(\hat y_{{\rm qce}, F})$ as
\begin{align*}
  \E_{\rm qce}''(\hat y_{{\rm qce}, F})[u,u] = \dots
  + \eps \sum_{\ell = 0}^{K-2} \big\{ A_F |u_\ell'|^2
     - \eps^2 \phi_{2F}'' |u_\ell''|^2 \big\}
  + \eps \sum_{\ell = K+3}^N A_F &\,|u_\ell'|^2 \\
  + \eps \big\{\phi_F'' + 2 \phi_{2F}'' + 2 \phi''(2F+\smfrac12\delta_1)\big\}
  &\,|u_{K-1}'|^2 \\
  + \eps \big\{\phi''(F+\smfrac12\delta_1)+ 3\phi''(2F+\smfrac12\delta_1)\big\}
  &\, |u_{K}'|^2 \\
  + \eps \big\{ \phi_F'' + \phi''(2F-\smfrac12\delta_1)
       + \phi''(2F + \smfrac12\delta_1)
  + 2 \phi_{2F}'' \big\} &\, |u_{K+1}'|^2 \\
  + \eps \big\{ \phi''(F-\smfrac12\delta_1) + \phi''(2F-\smfrac12\delta_1)
  + 4 \phi''(2F-\delta_1) \big\} &\, |u_{K+2}'|^2  \\
  - \eps^3 \big\{
  \phi''(2F+\smfrac12\delta_1) |u_{K-1}''|^2
  + \smfrac12 \phi''(2F+\smfrac12\delta_1) |u_K''|^2
  + \smfrac12 \phi''(2F-\smfrac12\delta_1) |u_{K+1}''|^2 \big\}. \hspace{-1.2cm} &
\end{align*}
We expand all terms containing $\delta_1$ and neglect all terms which
are of order $O(\delta^2)$ relative to $\phi''_F,$ which is the order
of magnitude of the coefficient of the diagonal term of $\E_{\rm
  qce}''(\hat{y}_{{\rm qce}, F}).$ For example, we have, for some $\vartheta
\in (0, 1)$,
\begin{displaymath}
 \frac{\phi''(2F+\smfrac12\delta_1)}{\phi''_F}
 = \frac{\phi_{2F}''}{\phi_F''} +
 \frac{\phi'''(2F + \vartheta \smfrac12\delta_1)}{\phi''_F} (\smfrac12\delta_1)
 = \frac{\phi_{2F}''}{\phi_F''} + O(\delta_3\delta_1),
\end{displaymath}
as $\delta_1,\delta_3 \to 0$. Thus, the $O(\delta_1)$ perturbation of a
second-neighbour term will not affect our final result. On the other
hand, expanding a nearest neighbour term gives
\begin{displaymath}
  \frac{\phi''(F + \smfrac12\delta_1)}{\phi''_F}
  = 1 + \frac{\phi_F'''}{\phi''_F} (\smfrac12\delta_1) + O(\delta_1^2),
\end{displaymath}
as $\delta_1 \to 0$.  Proceeding in the same fashion for the
remaining terms, we arrive at
\begin{align}
  \notag 
  \E_{\rm qce}''(\hat y_{{\rm qce}, F})[u,u] =~& \dots
  + \eps \sum_{\ell = 0}^{K-1} A_F |u_\ell'|^2
  - \eps^3 \sum_{\ell = 0}^{K-1} \phi_{2F}'' |u_\ell''|^2
  + \eps \sum_{\ell = K+3}^N A_F |u_\ell'|^2 \\
  \label{eq:qce:hatH_asymp}
  &  + \eps \big\{ A_F + (\smfrac12 \delta_1 \phi_F''' - \phi_{2F}'') \big\}
  |u_{K}'|^2 + \eps A_F |u_{K+1}'|^2  \\
  \notag 
  &
  + \eps \big\{ A_F - (\smfrac12\delta_1\phi_F''' - \phi_{2F}'') \big\} |u_{K+2}'|^2
  - \eps^3 \smfrac12 \phi_{2F}'' \big\{ |u_{K}''|^2 + |u_{K+1}''|^2 \big\} \\
  \notag 
  &  + O\big(\phi_F'' \delta^2 \|u'\|_{\ell^2_\eps}^2\big).
\end{align}

Clearly, our focus must be the coefficients of the terms $|u_K'|^2$
and $|u_{K+2}'|^2$, and in particular, on the quantity
\begin{equation}
  \label{sign}
  \smfrac12\delta_1\phi_F''' - \phi_{2F}'' = \frac{\phi_F''' \phi_{2F}'
    - 2 \phi_F'' \phi_{2F}''}{2 \phi_F''}.
\end{equation}
Depending on the sign of $\smfrac12\delta_1\phi_F''' - \phi_{2F}''<0,$
we see that the ``weakest bonds'' are either between atoms $K-1$ and
$K$ (as well as $-K+1$ and $-K$) or between atoms $K+1$ and $K+2$ (as
well as $-K-1$ and $-K-2$).

If $\smfrac12\delta_1\phi_F''' - \phi_{2F}''<0,$ we insert the test
function $w \in \Us$, defined by
\begin{displaymath}
  w_\ell' = \cases{
    (\smfrac12 \eps^{-1})^{1/2}, & \ell = K, \\
    -(\smfrac12 \eps^{-1})^{1/2}, & \ell = -K+1, \\
    0, & \text{otherwise},}
\end{displaymath}
into (\ref{eq:qce:hatH_asymp}) to obtain
\begin{equation}
  \label{eq:qce:stab_final}
  \begin{split}
    \inf_{\substack{u \in \Us \\ \| u'\|_{\ell^2_\eps} = 1} }
    \E_{\rm qce}''(\hat{y}_{{\rm qce}, F})[u,u]
    \leq~& \E_{\rm qce}''(\hat y_{{\rm qce}, F})[w,w] \\[-4mm]
    =~& A_F \Big\{ 1  + \frac{\phi_F''' \phi_{2F}'
      - 5\phi_F'' \phi_{2F}''}{2 A_F \phi_F''} + O(\delta^2) \Big\}.
  \end{split}
\end{equation}
Note that the constant $2$ in front of $\phi_F'' \phi_{2F}''$ was
replaced by $5$ due to the strain gradient terms in
(\ref{eq:qce:hatH_asymp}) which slightly stabilize the system.

If $\smfrac12\delta_1\phi_F''' - \phi_{2F}''>0,$ we use the alternative test
function $w \in \Us$, defined by
\begin{equation}
  \label{alt}
  w_\ell' = \cases{
    (\smfrac12 \eps^{-1})^{1/2}, & \ell = K+2, \\
    -(\smfrac12 \eps^{-1})^{1/2}, & \ell = -K-1, \\
    0, & \text{otherwise},}
\end{equation}
to test (\ref{eq:qce:hatH_asymp}), which gives
\begin{equation}
  \label{eq:qce:stab_final3}
  \begin{split}
    \inf_{\substack{u \in \Us \\ \| u'\|_{\ell^2_\eps} = 1} }
    \E_{\rm qce}''(\hat{y}_{{\rm qce}, F})[u,u]
    \leq~& \E_{\rm qce}''(\hat y_{{\rm qce}, F})[w,w] \\[-4mm]
    =~& A_F \Big\{ 1 - \frac{ \phi_F''' \phi_{2F}'-\phi_F'' \phi_{2F}''
      }{2 A_F \phi_F''} + O(\delta^2) \Big\}.
  \end{split}
\end{equation}
In this case, only a single strain gradient term affects the final
result, and therefore this correction is only small.

Due to the stabilizing effect of the strain gradient terms for our
perturbation, the right hand sides of~\eqref{eq:qce:stab_final}
and~\eqref{eq:qce:stab_final3} might both be bounded below by $A_F,$
so our estimate will involve a $\min$ over three terms.
Recalling that $y_{{\rm qce}, F} = \hat{y}_{{\rm qce}, F} + O(\delta^2),$ we obtain the
following result:

\begin{proposition}
  \label{th:qce:final_stab_result}
  There exist constants $\hat{\delta}$ and $\hat{C}$, which may depend
  on $\phi$ and its derivatives and on $F$ but not on $\eps$, such
  that, if $\delta \leq \hat{\delta}$, then
  \begin{equation}
    \label{eq:qce:stab_final2}
    \begin{split}
      \inf_{\substack{u \in \Us \\ \| u'\|_{\ell^2_\eps} = 1} }
      \E_{\rm qce}''(y_{{\rm qce}, F})[u,u] \leq \phi_F'' \Big( \min\!\Big\{\,&
       1+\frac{3\phi_{2F}''}{\phi_F''}
      \pm \Big( \frac{\phi_F'''\phi_{2F}'}{2|\phi_{F}''|^2}
		- \frac{3}{2} \frac{\phi_{2F}''}{\phi_F''} \Big),
       \frac{A_F}{\phi_F''} \Big\}
      + \hat{C} \delta^2 \, \Big).
    \end{split}
  \end{equation}
\end{proposition}
\begin{proof}
  The bounds (\ref{eq:qce:stab_final}) and (\ref{eq:qce:stab_final3})
  are rigorous provided $\delta$ is sufficiently small so that $F -
  \smfrac12 \delta_1$ is bounded away from zero. Moreover, if $\delta$
  is sufficiently small, then Lemma \ref{th:asymptotic_expansion}
  gives a rigorous bound for the error $\| (y_{{\rm qce}, F} - \hat{y}_{\rm
    qce})' \|_{\ell^\infty}$ which only adds an additional
  $O(\delta^2)$ error to the estimate.
\end{proof}

For typical interaction potentials, we would expect that $\phi_F''' <
0$ (as $\phi_F''$ is decreasing), that $\phi_{2F}' > 0$, and we have
already postulated that $\phi_F'' > 0$ and $\phi_{2F}'' < 0$. Thus,
the two terms in the numerator of the right hand side of \eqref{sign}
have opposing sign and may, in principle even cancel each
other. However, we have found in numerical tests that for typical
potentials such as the Morse or Lennard--Jones potentials the first
term is dominant, that is, $ \smfrac12\delta_1\phi_F''' -
\smfrac{3}{2} \phi_{2F}''< \phi''_{2F} $
and
\begin{displaymath}
  \min\!\Big\{\,
  1+\frac{3\phi_{2F}''}{\phi_F''}
  \pm \left( \frac{\phi_F'''\phi_{2F}'}{2|\phi_{F}''|^2}
  - \frac{3}{2} \frac{\phi_{2F}''}{\phi_F''} \right),
  \frac{A_F}{\phi_F''} \Big\}=
  1 + \frac{3}{2} \frac{\phi_{2F}''}{\phi_F''}
  + \frac{\phi_F'''\phi_{2F}'}{2|\phi_{F}''|^2}
\end{displaymath}
in Proposition~\ref{th:qce:final_stab_result}.

\begin{remark}
  Proposition \ref{th:qce:final_stab_result} as well as the subsequent
  discussion clearly shows that the spurious QCE instability is due to
  a combination of the effect of the``ghost force'' error and of the
  anharmonicity of the atomistic potential.
\end{remark}

\begin{remark}
\label{rmbounds}
  A variant of the analysis presented above shows that $\E_{\rm
    qce}''(y_F)$ is positive definite if and only if $A_F + \lambda_K
  \phi_{2F}'' > 0$ where $\smfrac 12 \le \lambda_K \le 1$.  The lower
  bound can be obtained using the test function~\eqref{alt} in the
  bilinear form $\E_{\rm qce}''(y_F)[u,\,u]$ given explicitly by
  \eqref{eq:Eqce_decomposition_final}, while the upper bound can be
  obtained from the estimate
  \begin{displaymath}
    \E_{\rm qce}''(y_F)[u,\,u]\ge (A_F+\phi_{2F}'')\| u'\|_{\ell^2_\eps}^2
    \qquad\text{ for all } u \in \Us,
  \end{displaymath}
  which also follows from~\eqref{eq:Eqce_decomposition_final} (see
  also \cite[Lemma 2.1]{Dobson:2008b}). Thus, the lower bound is
  related to the second term in ~\eqref{eq:qce:stab_final2} which we
  have noted above is generally greater than the first term, and we
  can conclude that the critical strain for QCE obtained by linearizing
  about $y_F,$ rather than the equilibrium solution $y_{{\rm qce}, F},$
  significantly underestimates the loss of stability (see also Figure
  \ref{fig:crit_strains_a}).

  The study of the positive-definiteness of $\E_{\rm qce}''(y_F)$ is
  relevant to the stability of the ghost-force correction iteration
  and is discussed in more detail in \cite{qcf.iterative}.
\end{remark}

\begin{remark}
  \label{rem:delta_discussion}
  While our rigorous results, Lemma \ref{th:asymptotic_expansion} and
  Proposition \ref{th:qce:final_stab_result}, are proven only for
  sufficiently small $\delta$, one usually expects that such
  asymptotic expansions have a wider range of validity than that
  predicted by the analysis. For this reason, we have neglected to
  give more explicit bounds on how small $\delta$ needs to be.

  Nevertheless, a relatively simple asymptotic analysis such as the
  one we have presented cannot usually give complete information near
  the onset of instability. Our aim was mainly to demonstrate that the
  inconsistency at the interface leads to a decreased stability of the
  QCE approximation when compared to the full atomistic model or the
  consistent QC approximations. We will see in Section \ref{sec:discus} that,
  if we use \eqref{eq:qce:stab_final2} to predict the onset of
  instability for QCE, then we observe a fairly significant loss of
  stability of the QCE approximation when compared to the full
  atomistic model. In numerical experiments, we will also see that the
  prediction given by (\ref{eq:qce:stab_final2}) is qualitatively
  fairly accurate for the Morse potential \alert{for a range of
  parameters that explores the dependence of our results on $\delta.$}
\end{remark}

\section{Prediction of the Limit Strain for Fracture Instability}
\label{sec:discus}
The deformation $y_F \in \Ys_F$ is an equilibrium of the atomistic
energy for {\em all $F > 0$}. However, it is established in
Proposition~\ref{th:ana:stab_a} that $y_F$ is {\em stable} if and only
if ${F < F_{\rm a}^*}$ where $F_{\rm a}^*$ is the solution of the equation
\begin{equation}
  \label{eq:defn_Fa_crit}
  \psi_{\rm a}(F_{\rm a}^*) := \phi''(F_{\rm a}^*) + (4 - \eps^2 \mu_\eps^2) \phi''(2F_{\rm a}^*) = 0.
\end{equation}
We call $F_{\rm a}^*$ the critical strain for the atomistic model. The goal
of the present section is to use the stability analyses of the
different QC approximations in Sections \ref{sec:analysis} and \ref{sec:ana:qce}
to investigate how well the critical strains for the different QC
approximations approximate that of the atomistic model.

In order to test our predictions against numerical values, we will use
the Morse potential
\begin{equation}
  \label{eq:defn_morse}
  \phi_\alpha(r) = e^{-2 \alpha (r - 1)} - 2e^{-\alpha(r-1)}=(e^{-\alpha(r-1)}-1)^2-1,
\end{equation}
where $\alpha \geq 1$ is a fixed parameter, and the Lennard--Jones
potential
\begin{equation*}
\phi_{\rm lj}(r) = \frac{1}{r^{12}} - \frac{2}{r^6}.
\end{equation*}

\subsection{Limit strain for the QCL and QNL approximations}
The critical strain $F_{\rm c}^*$ for the local QC approximation as well as
the QNL approximation (cf. Propositions \ref{th:ana:qcl_stab} and
\ref{th:ana:qnl_stab}) is the solution to the equation
\begin{equation*}
  \psi_{\rm c}(F_{\rm c}^*) := \phi''(F_{\rm c}^*) + 4 \phi''(2F_{\rm c}^*) = 0.
\end{equation*}

We note that the critical strain $F_{\rm c}^*$ for the QCL and QNL models is
independent of $N$ which is convenient for the following
analysis. Inserting $F_{\rm c}^*$ into (\ref{eq:defn_Fa_crit}) gives
\begin{displaymath}
  \psi_{\rm a}(F_{\rm c}^*)
  = \psi_{\rm c}(F_{\rm c}^*) - \eps^2\mu_\eps^2 \phi''(2 F_{\rm c}^*)
  = - \eps^2\mu_\eps^2\phi''(2 F_{\rm c}^*),
\end{displaymath}
and hence
\begin{displaymath}
  \psi_{\rm a}(F_{\rm a}^*) - \psi_{\rm a}(F_{\rm c}^*) = \eps^2\mu_\eps^2\phi''(2F_{\rm c}^*).
\end{displaymath}
A linearization of the left-hand side gives
\begin{displaymath}
  \psi_{\rm a}'(F_{\rm c}^*) ( F_{\rm a}^* - F_{\rm c}^*)
  = \eps^2\mu_\eps^2 \phi''(2 F_{\rm c}^*)
  + O( |F_{\rm a}^* - F_{\rm c}^*|^2 ).
\end{displaymath}
Noting that $\psi_{\rm a}'(F_{\rm c}^*) = \psi_{\rm c}'(F_{\rm c}^*) + O(\eps^2)$, we find
that the {\em relative error} satisfies
\begin{equation}
  \label{eq:rel_err_ac}
  \begin{split}
    \left| \frac{F_{\rm a}^* - F_{\rm c}^*}{F_{0} - F_{\rm c}^*} \right|
    =~& \eps^2  \left| \frac{\pi^2 \phi''(2F_{\rm c}^*)}{
        (\phi'''(F_{\rm c}^*)+8\phi'''(2F_{\rm c}^*))(F_{0} - F_{\rm c}^*)} \right|
    + O(\eps^4) \\
    :=~& \eps^2 C_{\rm err}(\phi) + O(\eps^4),
  \end{split}
\end{equation}
where $F_{0}$ is the energy-minimizing macroscopic deformation gradient which satisfies
\[
\frac{d\E_{\rm a}(y_F)}{dF}(F_0) =\phi'(F_0) + 2 \phi'(2F_0) = 0.
\]

\begin{table}
  \begin{tabular}{r||cccccc|c}
    $\phi$ & $\phi_2$ & $\phi_3$ & $\phi_4$ & $\phi_5$ & $\phi_6$ &
    $\phi_7$ & $\phi_{\rm lj}$ \\
    \hline
    $C_{\rm err}(\phi)$ & 1.0877 & 0.3796 & 0.1339 & 0.0485 & 0.0177 &
    0.0065 & 0.0635
  \end{tabular}

  \vspace{2mm}

  \caption{\label{tbl:errconsts} Numerical values of the error
    constant $C_{\rm err}(\phi)$ defined in~\eqref{eq:rel_err_ac},
    for various choices of $\phi$.}

  \vspace{-10mm}
\end{table}

In Table \ref{tbl:errconsts} we display numerical values of $C_{\rm
  err}(\phi)$ for the Morse potential $\phi = \phi_\alpha$, with
$\alpha = 2, \dots, 7$, and for the Lennard--Jones potential $\phi =
\phi_{\rm lj}$. We observe that the constant decays exponentially as
the stiffness increases, and that it is moderate even for very
soft interaction potentials ($C_{\rm err}(\phi_2) \approx 1.0877$).

\subsection{Limit strain for the QCE approximation}
In Section \ref{sec:ana:qce}, we have computed a rough estimate for
the coercivity constant of the QCE approximation. We argued that, for as long
as the second neighbour interaction is small in comparison to the
nearest neighbour interaction, we have the bound
\begin{displaymath}
  \inf_{\substack{u \in \Us \\ \|u'\|_{\ell^2_\eps} = 1}}
  \E_{\rm qce}''(y_{{\rm qce}, F})[u,u] \leq \phi_F'' \Big\{
  1 + \frac{3}{2} \frac{\phi_{2F}''}{\phi_F''} + \frac{\phi_F'''\phi_{2F}'}{2 |\phi_{F}''|^2} + O(\delta^2) \Big\}.
\end{displaymath}
Even though this bound will, in all likelihood, become invalid near
the critical strain, it is nevertheless reasonable to expect that
solving
\begin{equation}
  \label{eq:discus:critF_qce}
  \tilde\psi_{\rm qce}(\tilde{F}_{\rm qce}^*) :=
  \phi_F'' + \smfrac{3}{2} \phi_{2F}'' + \frac{\phi_F'''\phi_{2F}'}{2 \phi_{F}''} = 0,
\end{equation}
will give a good approximation for the exact critical strain,
$F_{\rm qce}^*$. The latter is, loosely speaking, defined as the maximal
strain $F > 0$ for which a stable ``elastic'' equilibrium of
$\E_{\rm qce}$ exists in $\Ys_F$. A deformation $y$ can be called elastic
if $y_\ell'= O(1)$ for all $\ell$, as opposed to fractured if
$y_{\ell_0}' = O(N)$ for some $\ell_0$.

We could use the same argument as in the previous subsection to obtain a
representation of the error; however, since $\tilde F_{\rm qce}^*$ depends
only on $F$ but not on $\eps$ we can simply solve for $\tilde
F_{\rm qce}^*$ directly.

For the Morse potential \eqref{eq:defn_morse}, with stiffness
parameter $2 \leq \alpha \leq 7$, we have computed both $F_{\rm
  qce}^*$ (for $N = 40, K= 10$ as well as for $N = 100, K = 20$) and
$\tilde F_{\rm qce}^*$ numerically and have plotted these critical
strains in Figure \ref{fig:crit_strains_a}, comparing them against
$F_0$ and $F_{\rm c}^*$. We have also included the critical strain
$F_{\rm qce}^{y_F}$, below which $\E_{\rm qce}''(y_F)$ is positive
definite, to demonstrate that it bears no relation to the stability or
instability of the QCE approximation. We discuss $F_{\rm qce}^{y_F}$ in
detail in \cite{qcf.iterative} where we argue that it describes
the stability of the ghost-force correction scheme.

In Figure \ref{fig:crit_strains_b}, we plot the relative errors
\begin{equation*}
  \alpha \mapsto \left| \frac{F_{\rm qce}^*(\alpha) - F_{\rm c}^*(\alpha)}{
      F_{\rm c}^*(\alpha) - F_0(\alpha)} \right|
  \quad \text{and} \quad
  \alpha \mapsto \left| \frac{\tilde F_{\rm qce}^*(\alpha) - F_{\rm c}^*(\alpha)}{
      F_{\rm c}^*(\alpha) - F_0(\alpha)} \right|.
\end{equation*}

\begin{figure}
  \begin{center}
    \includegraphics[height=6.6cm]{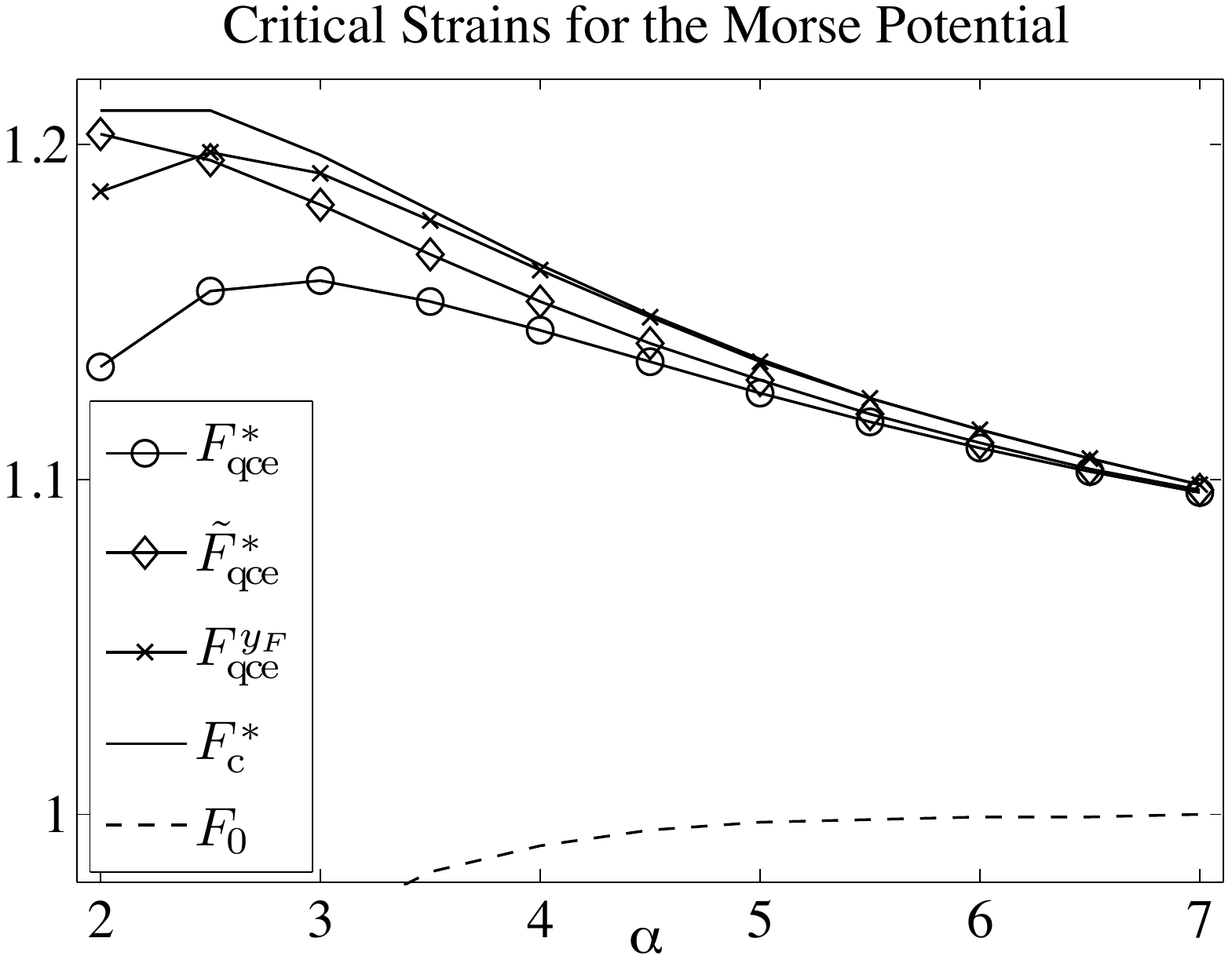}
    \caption{\label{fig:crit_strains_a} Critical strains $F_{\rm
        qce}^*,\ \tilde F_{\rm qce}^*,\ F_{\rm qce}^{y_F},\ F_{\rm c}^*$
      and the equilibrium strain $F_0$, computed for the Morse
      potential \eqref{eq:defn_morse} with varying $\alpha$.
      The critical strains for the QCE Hessian, $F_{\rm qce}^*,$
      are computed with $N = 40$ and
      $K = 10$.  The approximation, $\tilde F_{\rm qce}^*,$ is computed
      using the asymptotic approximation~\eqref{eq:discus:critF_qce}.
      The strain $F_{\rm qce}^{y_F}$ is the critical strain
      at which $\E_{\rm qce}''(y_F)$ is no longer positive definite.}
  \end{center}
\end{figure}

\begin{figure}
  \begin{center}
    \includegraphics[height=6.6cm]{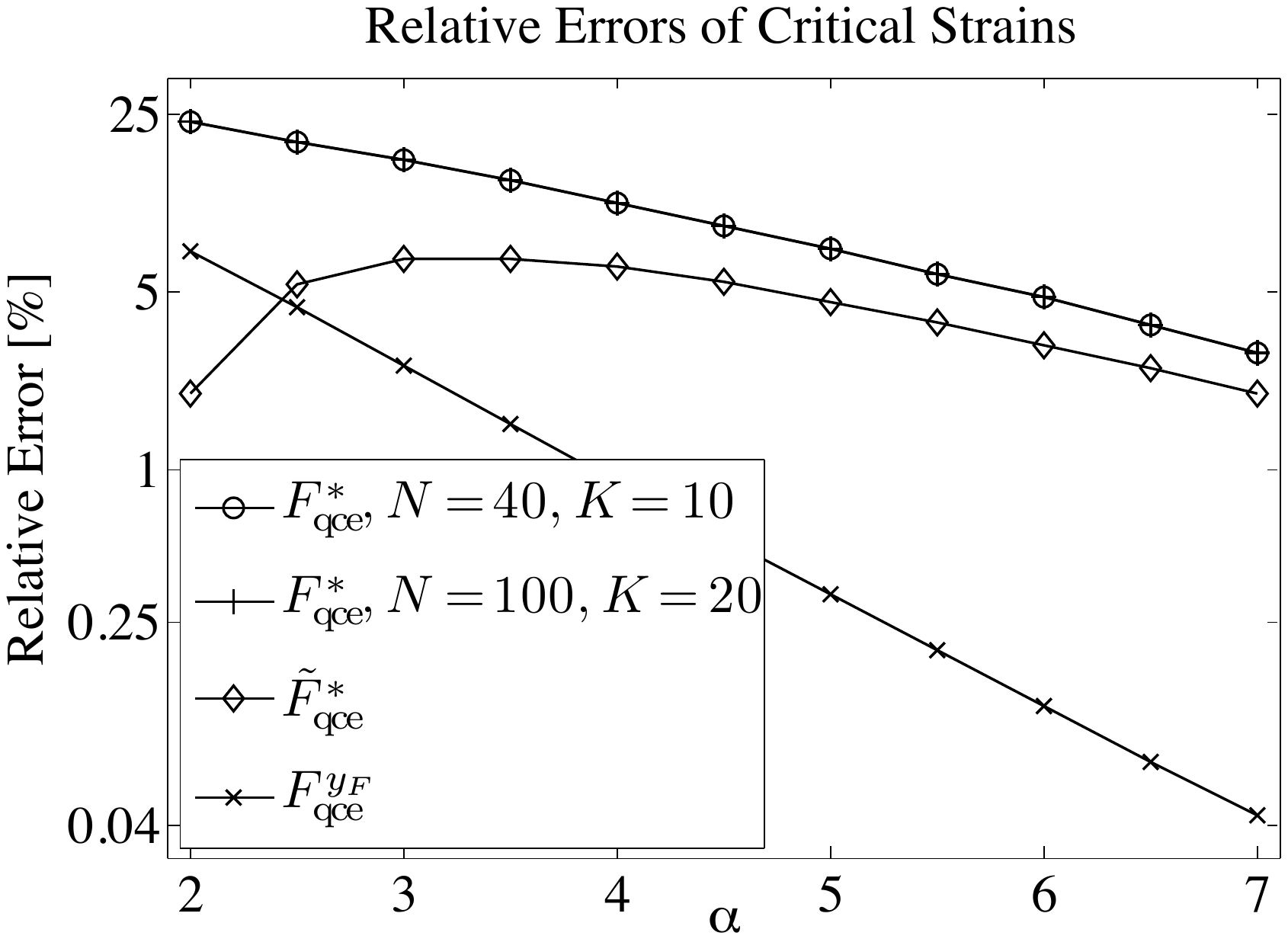}
    \caption{\label{fig:crit_strains_b} Relative errors of the
      critical strains (computed and predicted) for the QCE approximation
      against the critical strains of the QCL/QNL approximation. The errors
      are computed explicitly for $N = 40, K = 10$ as well as for $N =
      100, K = 20$, using the Morse potential \eqref{eq:defn_morse}
      with varying $\alpha$.  These two curves are very close and may
      be hard to distinguish.  Additionally, we show the critical
      strain for loss of positive definiteness of $\E_{\rm
        qce}''(y_F),$ which does not predict the loss of stability
      that the QCE experiences correctly for any parameter value.}
  \end{center}
\end{figure}

We observe that the prediction for the critical strain, as well as the
prediction for the relative error, obtained from our asymptotic
analysis is insufficient for very soft potentials but becomes fairly
accurate with increasing stiffness. In particular, it provides a good
prediction of the relative errors for the critical strains for $\alpha
\geq 3.5$.

For a correct interpretation of our results, we must first of all note
that the relative errors for the critical strains decay exponentially
with increasing stiffness $\alpha$. While, for small $\alpha$ (soft
potentials) the error is quite severe, one could argue that it is
insignificant (i.e., well below 10\%) for moderately large $\alpha$
(stiff potentials). However, our point of view is that, by a careful
choice of the atomistic region one should be able to control this
error, as is the case for consistent QC approximations such as QNL. For the
QCE approximation, this is impossible: the error in the critical strain is
{\em uncontrolled}.

\section*{Conclusion}

We propose sharp stability analysis as a theoretical criterion for
evaluating the predictive capability of atomistic-to-continuum
coupling methods.  Our results show that a sharp stability analysis is
as important as a sharp truncation error (consistency) analysis for
the evaluation of atomistic-to-continuum coupling methods, and
provides a new means to distinguish the relative merits of the various
methods.  Our results also provide an approach to establish a
theoretical basis for the conclusions of the benchmark numerical tests
reported in ~\cite{Miller:2008}, in particular for the poor
performance of the QCE approximation in predicting the movement of a
dipole of Lomer dislocations under applied shear.

\alert{
Of course, the simple one-dimensional situation that we have
considered here cannot nearly capture the complexity of
atomistic-to-continuum coupling methods in 2D/3D. Even the much
simpler question of whether QCL (the Cauchy-Born continuum model without
coupling) can correctly predict bifurcation points becomes much more
difficult since it is possible, in general, that the stability region
for the Cauchy--Born model is much larger than that for atomistic
model \cite{E:2007a}. However, in many interesting situations this
effect does not occur \cite{HudsOrt:a}, and it is an interesting
question to characterize these. Concerning the stability of the coupling
mechanism, no rigorous results are available in 2D/3D. Until such an
analysis is available, we propose that careful numerical experiments
should be performed, which experimentally investigate the stability
properties of atomistic-to-continuum coupling methods.
}


\appendix

\section{Representations of $\E_{\rm qce}'$ and $\E_{\rm qce}''$}
\label{sec:app_qce1}
Our aim in this section is to derive useful representations for the
first and second variations $\E_{\rm qce}'(y)$ and $\E_{\rm qce}''(y)$ of the
QCE energy functional. For notational convenience, we will only write
out terms in the right half of the domain $\{-N+1, \dots, N\}$,
indicating the remaining terms (which can be obtained from symmetry
considerations) by dots. For example, we write
\begin{align*}
  \E_{\rm qce}(y) =~& \dots + \eps \sum_{\ell = 0}^{N} \phi(y_\ell')
  + \eps \sum_{\ell = 0}^{K-1} \phi(y_\ell' + y_{\ell+1}')
  + \eps \sum_{\ell = K+2}^N \phi(2 y_\ell') \\
  & + \smfrac{\eps}{2} \phi(y_K' + y_{K+1}')
  + \smfrac{\eps}{2} \phi(y_{K+1}' + y_{K+2}')
  + \smfrac{\eps}{2} \phi(2y_{K+1}').
\end{align*}

The first variation is a linear form on $\Us$, given by
\begin{align*}
  \E_{\rm qce}'(y)[u] = \dots + \eps \sum_{\ell = 0}^N \phi'(y_\ell') u_\ell'
  + \eps \sum_{\ell = 0}^{K-1} \phi'(y_\ell'+y_{\ell+1}')(u_\ell' + u_{\ell+1}')& \\
  + \smfrac{\eps}{2} \phi'(y_K' + y_{K+1}')(u_K'+u_{K+1}')
  + \smfrac{\eps}{2} \phi'(y_{K+1}' + y_{K+2}')(u_{K+1}'+u_{K+2}') & \\
  + \smfrac{\eps}{2} \phi'(2y_{K+1}') (2 u_{K+1}')
  + \eps \sum_{\ell = K+2}^N \phi'(2y_\ell')(2u_\ell')&.
\end{align*}
Collecting terms related to element strains $u_\ell'$, we obtain
\begin{align}
  \label{eq:app:DEqce}
  \begin{split}
  \E_{\rm qce}'(y)[u] = \dots +
  \eps \sum_{\ell = 0}^{K-1} \big\{ \phi'(y_\ell')
  + \phi'(y_{\ell-1}'+y_\ell')+\phi'(y_\ell'+y_{\ell+1}')\big\}~& u_\ell' \\
  + \eps \big\{ \phi'(y_K') + \phi'(y_{K-1}'+y_K')
  + \smfrac12 \phi'(y_{K}'+y_{K+1}')\big\}~& u_K' \\
  + \eps \big\{ \phi'(y_{K+1}') +  \smfrac12 \phi'(y_{K}'+y_{K+1}')
  +\smfrac12 \phi'(y_{K+1}'+y_{K+2}') + \phi'(2y_{K+1}') \big\} ~& u_{K+1}'\\
  + \eps \big\{ \phi'(y_{K+2}' + \smfrac12 \phi'(y_{K+1}'+y_{K+2}')
  + 2 \phi'(2y_{K+2}') \big\} ~& u_{K+2}' \\
  + \eps \sum_{\ell = K+3}^N \big\{ \phi'(y_\ell')
  + 2 \phi'(2y_\ell') \big\} ~& u_{\ell}'.
  \end{split}
\end{align}

Similarly, the second variation can be written in the form
\begin{equation}
  \label{eq:app:qce_hess_1}
  \begin{split}
    \E_{\rm qce}''(y)[u,u] = \dots
    + \eps \sum_{\ell = 0}^N \phi''(y_\ell') |u_\ell'|^2
    + \eps \sum_{\ell = 0}^{K-1} \phi''(y_\ell'+y_{\ell+1}')
    |u_\ell' + u_{\ell+1}'|^2& \\
    + \smfrac{\eps}{2} \phi''(y_K'+y_{K+1}')|u_K'+u_{K+1}'|^2
    + \smfrac{\eps}{2} \phi''(y_{K+1}'+y_{K+2}') |u_{K+1}'+u_{K+2}'|^2 & \\
    + \smfrac{\eps}{2} \phi''(2 y_{K+1}') |2u_{K+1}'|^2
    + \eps\sum_{\ell = K+2}^N \phi''(2y_\ell') |2u_\ell'|^2&.
  \end{split}
\end{equation}
Using (\ref{eq:rewrite_nnn}) to replace all second-neighbour terms in
(\ref{eq:app:qce_hess_1}), we obtain the alternative representation
\begin{align}
  \label{eq:app:qce_hess_2}
  \E_{\rm qce}''(y)[u,u] = \dots
  + \eps \sum_{\ell = 0}^{K-1} \big[
  \phi''(y_\ell') + 2 \phi''(y_{\ell-1}'+y_\ell')
  + 2\phi''(y_{\ell}'+y_{\ell+1}') \big]& |u_\ell'|^2 \\
  \notag
  + \eps \big[ \phi''(y_K') + 2 \phi''(y_{K-1}'+y_K')
  + \phi''(y_K'+y_{K+1}') \big]& |u_K'|^2 \\
  \notag
  + \eps \big[ \phi''(y_{K+1}') + \phi''(y_{K}'+y_{K+1}')  +
  \phi''(y_{K+1}'+y_{K+2}')  + 2\phi''(2y_{K+1}') \big]& |u_{K+1}'|^2 \\
  \notag
   + \eps \big[ \phi''(y_{K+2}') + \phi''(y_{K+1}'+y_{K+2}')
  + 4 \phi''(2y_{K+2}') \big]& |u_{K+2}'|^2 \\[-1mm]
  \notag
  + \eps \sum_{\ell = K+3}^{N} \big[ \phi''(y_\ell')
  + 4 \phi''(2 y_\ell')\big]& |u_\ell'|^2 \\[-4mm]
  \notag
  - \eps^3 \sum_{\ell = 0}^{K-1} \phi''(y_\ell'+u_{\ell+1}') |u_\ell''|^2
  - \smfrac12 \eps^3 \big\{ \phi''(y_K'+y_{K+1}') |u_K''|^2
  + \phi''(y_{K+1}'+y_{K+2}') |u_{K+1}''|^2 \big\}. \hspace{-2cm} &
\end{align}
While somewhat unwieldy at first glance, this representation is
particularly useful for the stability analysis in Section
\ref{sec:ana:qce}.

\section{Proof of Lemma \ref{th:asymptotic_expansion}}
\label{sec:app_qce_eq}
In this section, we complete the proof of Lemma
\ref{th:asymptotic_expansion} which was merely hinted at in the main
text of Section \ref{sec:ana:qce}. Recall that $\hat{y}_{{\rm qce}, F} =
y_F + \delta_1 \hat{g}$ where $\hat{g}$ is given by
\eqref{eq:qce:defn_uhat}, and recall, moreover, that $\hat{y}_{\rm
  qce}$ solves the linear system
\begin{equation}
  \label{eq:app_qce:defnyhat}
  \phi_F'' \< (\hat y_{{\rm qce}, F} - y_F)', u' \> = \phi_{2F}' \< \hat g', u' \>
  = - \E_{\rm qce}'(y_F)[u] \qquad \text{ for all } u \in \Us.
\end{equation}
Our strategy is to prove that $\hat{y}_{{\rm qce}, F}$ has a residual of order
O($\delta_1^2 + \delta_1\delta_2$) and that $\E_{\rm qce}''(\hat y_{{\rm qce}, F})$
is an isomorphism between suitable function spaces. We will then apply
a quantitative inverse function theorem to prove the existence of a
solution $y_{{\rm qce}, F}$ of the QCE criticality condition
\eqref{eq:qce:nonlin} which is ``close'' to $\hat y_{{\rm qce}, F}$. Before we
embark on this analysis, we make several comments and introduce some
notation that will be helpful later on.

To ensure that $\E_{\rm qce}$ is sufficiently differentiable in a
neighbourhood of $\hat y_{{\rm qce}, F}$ we only need to assume that $F > 0$
and that $\delta_1$ is sufficiently small, e.g., $\delta_1 \leq F$. In
that case, $\E_{\rm qce}$ is three times differentiable at $y$ for any $y
\in \Ys_F$ such that $\|y' - \hat y_{{\rm qce}, F}' \|_{\ell^\infty} <
\smfrac12 \delta_1$.

We will interpret $\E_{\rm qce}'$ as a nonlinear operator from
$\Us^{1,\infty}$ to $\Us^{-1,\infty}$ which are, respectively, the
spaces $\Us$ and $\Us^*$ endowed with the Sobolev-type norms,
\begin{displaymath}
  \| u \|_{\Us^{1,\infty}} = \|u'\|_{\ell^\infty} \quad \text{for } u \in \Us,
  \quad \text{and}  \quad
  \| T \|_{\Us^{-1,\infty}} =  \sup_{\substack{v \in \Us \\\|v'\|_{\ell^1_\eps} = 1}} T[v]
  \quad \text{for } T \in \Us^*.
\end{displaymath}
Consequently, for $y \in \Ys_F$, $\E_{\rm qce}''(y)$ can be understood as
a linear operator from $\Us^{1,\infty}$ to $\Us^{-1,\infty}$.

Our justification for defining $\hat{y}_{{\rm qce}, F}$ as we did in
\eqref{eq:app_qce:defnyhat} is the bound
\begin{equation}
  \label{eq:D2EyF_expansion}
  \big| \E_{\rm qce}''(y_F)[u,v] - \phi_F'' \<u', v'\> \big|
  \leq \phi_F'' c_1 \delta_2 \|u'\|_{\ell^\infty_\eps} \|v'\|_{\ell^1_\eps}
  \quad \text{ for all } u, v \in \Us,
\end{equation}
where $c_1 = 5$, which follows from
\eqref{eq:Eqce_decomposition_final}. We can formulate this bound
equivalently as
\begin{equation}
  \label{eq:D2EyF_expansion_2}
  \| \E_{\rm qce}''(y_F) - \phi_F'' L_1 \|_{L(\Us^{1,\infty},\ \Us^{-1,\infty})}
  \leq \phi_F'' c_1 \delta_2,
\end{equation}
where $L_1 : \Us \to \Us^*$ is given by
\begin{displaymath}
  L_1 (u)[ v ] = \< u', v' \> \qquad \text{ for all } u, v \in \Us.
\end{displaymath}
We also remark that $L_1 : \Us^{1,\infty} \to \Us^{-1,\infty}$ is an
isomorphism, uniformly bounded in $N$, more precisely,
\begin{equation}
  \label{eq:qceapp:L1_iso}
  \| L_1^{-1} \|_{L(\Us^{-1,\infty},\ \Us^{1,\infty})} \leq 2.
\end{equation}
This result follows, for example, as a special case of
\cite[Eq. (36)]{Ortner:2008a} or \cite[Eq. (5.2)]{dobs-qcf2}, and is
also contained in \cite{doblusort:qcf.stab}.

We are now ready to estimate the residual of $\hat{y}_{{\rm qce}, F}$.
Expanding $\E_{\rm qce}'(\hat y_{{\rm qce}, F})$ to first order gives
\begin{equation}
  \label{eq:DEqce(haty)}
  \begin{split}
    \E_{\rm qce}'(\hat y_{{\rm qce}, F})[v] =~& \big\{ \E_{\rm qce}'(y_F)[v]
    + \delta_1 \E_{\rm qce}''(y_F)[\hat g, v] \big\} \\
    & + \delta_1 \int_0^1 \big\{
    \E_{\rm qce}''(y_F + t \delta_1 \hat g)[\hat g, v]
    - \E_{\rm qce}''(y_F)[\hat g, v] \big\} \dt.
  \end{split}
\end{equation}
We will estimate the two groups on the right-hand side of
\eqref{eq:DEqce(haty)} separately. Using \eqref{eq:qce:ghost_force}
and \eqref{eq:D2EyF_expansion}, we obtain
\begin{equation}
  \label{app:qce:resest_1}
  \begin{split}
    \big|\E_{\rm qce}'(y_F)[v] + \delta_1 \E_{\rm qce}''(y_F)[\hat g, v]\big|
    =~& \delta_1 \big| - \phi_F'' \< \hat g', v'\> + \E_{\rm qce}''(y_F)[\hat g, v]
    \big| \\
    \leq~& \phi_F'' c_1 \delta_1  \delta_2 \| \hat g'\|_{\ell^\infty_\eps}
    \|v'\|_{\ell^1_\eps} \qquad \text{ for all } v \in \Us.
  \end{split}
\end{equation}

To estimate the second group in \eqref{eq:DEqce(haty)} we simply use
the regularity of the interaction potential (we assumed that $\phi \in
\CC^3(0, +\infty)$) and H\"{o}lder's inequality to obtain
\begin{equation}
  \label{app:qce:resest_2}
  \big| \E_{\rm qce}''(y_F + t \delta_1 \hat g)[\hat g, v]
  - \E_{\rm qce}''(y_F)[\hat g, v] \big| \leq \phi_F'' c_2 t \delta_1
  \| \hat g'\|_{\ell^\infty}^2 \|v'\|_{\ell^1_\eps},
\end{equation}
where $(\phi_F'' c_2)$ is a local Lipschitz constant for $\phi''$,
that is, there exists a universal constant $\hat{c}_2$ such that
\begin{displaymath}
  c_2 = \hat c_2 \sup_{|r| \leq \smfrac12\delta_1} \frac{\max(|\phi'''(F+r)|,
    |\phi'''(2(F+r))|)}{\phi_F''}.
\end{displaymath}
In particular, if $\delta_1$ is sufficiently small then we may assume
that
\begin{displaymath}
  c_2 = 2 \hat{c}_2 \frac{\max( |\phi_F'''|, |\phi_{2F}'''|)}{\phi_F''}.
\end{displaymath}

Inserting \eqref{app:qce:resest_2} and \eqref{app:qce:resest_1} into
(\ref{eq:DEqce(haty)}), and using the fact that $\|\hat
g'\|_{\ell^\infty} = \smfrac12$, we obtain the
$\Us^{-1,\infty}$-residual estimate
\begin{equation*}
  \| \E_{\rm qce}'(\hat y_{{\rm qce}, F}) \|_{\Us^{-1,\infty}}
  \leq \phi_F''(\smfrac12 c_1 \delta_1 \delta_2 + \smfrac18 c_2 \delta_1^2).
\end{equation*}

Next, we estimate $\| \E_{\rm qce}''(\hat y_{{\rm qce}, F})^{-1} \|_{L(\Us^{-1,\infty},\
  \Us^{1,\infty})}$.  Using \eqref{eq:D2EyF_expansion_2} and a similar
argument as for \eqref{app:qce:resest_2} gives
\begin{align*}
  \| \E_{\rm qce}''(\hat y_{{\rm qce}, F}) - \phi_F'' L_1 \|_{L(\Us^{1,\infty},\ \Us^{-1\infty})}
  \leq~&
  \| \E_{\rm qce}''(\hat y_{{\rm qce}, F}) - \E_{\rm qce}''(y_F) \|_{L(\Us^{1,\infty},\ \Us^{-1\infty})} \\
  & + \| \E_{\rm qce}''(y_F) - \phi_F'' L_1 \|_{L(\Us^{1,\infty},\ \Us^{-1\infty})} \\
  \leq~& \phi_F'' (\smfrac12 c_2 \delta_1 + c_1 \delta_2).
\end{align*}
Moreover, from \eqref{eq:qceapp:L1_iso}, we deduce that
\begin{displaymath}
  \| (\phi_F'' L_1)^{-1} \|_{L(\Us^{-1,\infty},\ \Us^{1,\infty})} \leq \frac{2}{\phi_F''}.
\end{displaymath}

A standard result of operator theory states that if $X, Y$ are Banach
spaces and $T, S : X \to Y$ are bounded linear operators with $T$
being invertible and satisfying $\|S - T\| < 1 / \|T^{-1}\|,$ then $S$
is invertible and
\begin{displaymath}
  \|S^{-1} \| \leq \frac{ \|T^{-1}\|}{1 - \|T^{-1}\| \|S-T\|}.
\end{displaymath}
In our case, setting $T = \phi_F'' L_1$ and $S = \E_{\rm qce}''(\hat
y_{{\rm qce}, F})$, this translates to
\begin{displaymath}
  \| \E_{\rm qce}''(\hat y_{{\rm qce}, F})^{-1} \|_{L(\Us^{-1,\infty},\ \Us^{1,\infty})}
  \leq \frac{2}{\phi_F''(1 - \smfrac12 c_2 \delta_1 - c_1 \delta_2)},
\end{displaymath}
provided that the denominator is positive. Thus, for $\delta_1,
\delta_2$ sufficiently small, we obtain the bound
\begin{equation*}
  \| \E_{\rm qce}''(\hat y_{{\rm qce}, F})^{-1} \|_{L(\Us^{-1,\infty},\ \Us^{1,\infty})}
  \leq \frac{4}{\phi_F''}.
\end{equation*}

We now apply the following version of the inverse function theorem.

\begin{lemma}
  \label{th:inverse_fcn_thm}
  Let $X, Y$ be Banach spaces, $U$ an open subset of $X$, and let $F :
  U \to Y$ be Fr\'{e}chet differentiable. Suppose that $x_0 \in U$
  satisfies the conditions
  \begin{align*}
    &  \| F(x_0) \|_{Y} \leq \eta, \quad
    \| F'(x_0)^{-1} \|_{L(Y,X)} \leq \sigma^{-1}, \\
    & \overline{B_X(x_0, 2\eta\sigma^{-1})} \subset U, \\
    & \| F'(x_1) - F'(x_2) \|_{L(X,Y)} \leq L \| x_1 - x_2 \|_X \quad
    \text{for} \quad \|x_j - x_0\|_X \leq 2 \eta \sigma^{-1}, \\
    & \text{ and } 2 L \sigma^{-2} \eta < 1,
  \end{align*}
  then there exists $x \in X$ such that $F(x) = 0$ and $\|x - x_0\|_X
  \leq 2\eta\sigma^{-1}$.
\end{lemma}
\begin{proof}
  The result follows, for example, by applying Theorem 2.1 in
  \cite{ortner_apostex} with the choices $R = 2 \eta\sigma^{-1}$,
  $\omega(x_0, R) = LR$ and $\bar\omega(x_0, R) = \frac12 L
  R^2$. Similar results can be obtained by tracking the constants in
  most proofs of the inverse function theorem, and assuming local
  Lipschitz continuity of $F'$.
\end{proof}

\medskip \noindent For our purposes, we set $X = \Us^{1,\infty}$, $Y =
\Us^{-1,\infty}$, $F(u) = \E_{\rm qce}'(\hat y_{{\rm qce}, F} + u)$, and $x_0 =
0$. Assuming that $\delta_1, \delta_2$ are sufficiently small, our
previous analysis gives the residual and stability estimates
\begin{displaymath}
  \eta = \phi_F'' (\smfrac12 c_1\delta_1\delta_2 + \smfrac18 c_2 \delta_1^2)
  \quad \text{and} \quad \sigma = \smfrac14 \phi_F'',
\end{displaymath}
and, in particular,
\begin{displaymath}
  2\eta\sigma^{-1} = 4 c_1 \delta_1\delta_2 +  c_2 \delta_1^2.
\end{displaymath}

To ensure that $\overline{B_{\Us^{1,\infty}}(0, 2\eta\sigma^{-1})}$ remains
within the region of differentiability of $F$, that is, to ensure that
$(\hat y_{{\rm qce}, F} + u)_\ell' > 0$ for $\|u'\|_{\ell^\infty} \leq
2\eta\sigma^{-1}$, it is clearly enough to assume that $\delta_1$ and
$\delta_2$ are sufficiently small.

A modification of \eqref{app:qce:resest_2} then allows the choice $L =
2 \phi_F'' c_2$ for the local Lipschitz constant.

Thus, the condition ensuring the existence of a solution $y_{{\rm qce}, F}$ of
\eqref{eq:qce:nonlin} becomes
\begin{displaymath}
  4 L \sigma^{-2} \eta
  = 64 c_1c_2 \delta_1\delta_2 + 16 c_2^2 \delta_1^2
  < 1,
\end{displaymath}
which is satisfied, once again, if we assume that $\delta_1$ and
$\delta_2$ are sufficiently small. An application of Lemma
\ref{th:inverse_fcn_thm} concludes the proof of Lemma
\ref{th:asymptotic_expansion}.

\end{document}